\documentclass[12pt,a4paper,leqno]{amsart}
\usepackage{pifont}
\usepackage{wasysym}
\usepackage{bbm}
\usepackage{extarrows}
\usepackage{amssymb,xspace} 
\usepackage{amstext}
\usepackage{pdfsync}
\usepackage{hyperref} 
\usepackage[mathscr]{eucal} 
\theoremstyle{plain}
\usepackage{amsbsy,amssymb,amsfonts,latexsym,eucal,amscd}
\usepackage[all]{xy}
\usepackage{fourier-orns}

\newcommand{\tens}[1][]{\mathbin{\otimes_{\raise1.5ex\hbox to-.1em{}{#1}}}}

\newtheorem{theorem}{Theorem}[section]

\newtheorem{lemma}[theorem]{Lemma}

\newtheorem{proposition}[theorem]{Proposition}

{\theoremstyle{definition}}

{\theoremstyle{definition}}

{\theoremstyle{definition}}

{\theoremstyle{definition}}

{\theoremstyle{definition}\newtheorem{definition}[theorem]{Definition}}

{\theoremstyle{definition}}

{\theoremstyle{definition}\newtheorem{remark}[theorem]{Remark}}

\author{Julien Grivaux}

\address{Sorbonne Université \\ 
Institut de Mathématiques de Jussieu-Paris Rive Gauche \\
Case 247, 4 place Jussieu \\
F-75005, Paris, France.}

\email{jgrivaux@math.cnrs.fr}
\title{Derived intersections and free dg-Lie algebroids}

\begin{document}

\thanks{\textit{To Professor Masaki Kashiwara, with admiration and respect}}

\begin{abstract}
We study free dg-Lie algebroids over arbitrary derived schemes, and compute their universal enveloping and jet algebras. We also introduce derived twisted connections, and relate them with lifts on twisted square zero extensions. This construction allows us to provide new conceptual approaches of existing results concerning the derived deformation theory of subschemes.
\end{abstract}
\vspace*{1.cm}
\maketitle
\tableofcontents

\section{Introduction}
Let $\mathbf{k}$ be a field of characteristic zero, let $X$ and $Y$ to smooth $\mathbf{k}$-schemes, and assume that $X$ is a closed subscheme of $Y$. The derived self-intersection of $X$ in $Y$ is a derived $\mathbf{k}$-scheme that can be realized concretely as $X$ endowed with a sheaf of dg-algebras whose underlying object in the derived category of $X$ is the derived tensor product $\mathcal{O}_X \overset{\mathbb{L}}{\otimes}_{\mathcal{O}_Y} \mathcal{O}_X$. Therefore, computing this object as well as its dg structure on it is of high interest.
\par \smallskip
The case that has been the most studied at first is the case of the diagonal embedding. In that case the corresponding derived intersection is called the derived loop space of $X$, and can be thought as the derived algebraic version of the space of parametrized loops of a topological space. It has the structure of a derived group scheme over $X$, and the underlying derived scheme can be interpreted as the total space of the derived bundle $\mathrm{T}_X[-1]$. This is a geometric way to express the geometric Hochschild-Kostant-Rosenberg isomorphism as an isomorphism between $\mathcal{O}_X \overset{\mathbb{L}}{\otimes}_{\mathcal{O}_{X \times X}} \mathcal{O}_X$ and $\mathrm{sym}(\Omega^1_X[1])$. The derived group scheme structure is encoded by a Lie algebra on $\mathrm{T}_X[-1]$, that has been discovered by Kapranov \cite{kapranov_rozansky-witten_1999} and Markarian \cite{markarian_atiyah_2009}. The Lie bracket turns out to be the Atiyah class of the tangent bundle $\mathrm{T}_X$. Besides, the universal enveloping algebra of this Lie algebra is isomorphic to the derived Hoschild homology complex $\mathcal{RH}om_{\mathcal{O}_{X \times X}}(\mathcal{O}_X, \mathcal{O}_X)$, as proven in \cite{markarian_atiyah_2009}, \cite{ramadoss_big_2008}, and more recently by a new method in \cite{calaque_ext_2017}.
\par \smallskip
Various attempts have been made by several authors in order to have a grasp on more general situations, starting with \cite{arinkin_when_2012}: the authors prove that the Hochschild-Kostant-Rosenberg doesn't always hold in the case of a general pair $(X, Y)$: a necessary and sufficient condition for this is that the conormal bundle $\mathrm{N}^*_{X/Y}$ extends to a locally free sheaf at the first order. A particular case of this setting, namely the case of quantized cycles, is developed in  \cite{grivaux_hochschild-kostant-rosenberg_2014}, building on a 1992 unpublished letter from Kashiwara to Schapira. However, even in the quantized case, the ext algebra $\mathcal{RH}om_{\mathcal{O}_Y}(\mathcal{O}_X, \mathcal{O}_X)$ is not easy to understand. The recent work \cite{calaque_ext_2017} allows to describe concretely this algebra for tame quantized cycles, which is a special class of quantized cycles discovered in \cite{yu_todd_2015}.
\par \smallskip
Let us now focus on non-quantized cycles. This is a more difficult task that has been completed in the groundbreaking paper \cite{calaque_lie_2014}, and independently in \cite{MR3570155}. The main idea in \cite{calaque_lie_2014} is that the self-derived intersection of $X$ in $Y$ is an algebraic counterpart of the set of continuous paths in $Y$ whose beginning and end points lie in $X$. Hence, it is no longer a derived group scheme, but there should be a groupoid structure given by the composition of paths.
This is indeed the case: the shifted normal bundle $\mathrm{N}_{X/Y}[-1]$ is naturally $\mathrm{L}_{\infty}$ dg-Lie algebroid, the anchor map being given by the extension class of the normal exact sequence of the pair $(X, Y)$. The enveloping algebra and the jet algebra of this Lie algebroid are isomorphic to $\mathcal{RH}om_{\mathcal{O}_Y}(\mathcal{O}_X, \mathcal{O}_X)$ and  $\mathcal{O}_X \overset{\mathbb{L}}{\otimes}_{\mathcal{O}_Y} \mathcal{O}_X$ respectively. This result gives a full understanding of the situation except that the homotopy-Lie algebroid structure is delicate to compute practically. On the other hand, the self-intersection of $X$ into its first formal neighborhood is computed in \cite{bigpaper}. 
\par \smallskip
Our goal is to provide among other things an alternative conceptual approach of the main result of \cite{bigpaper} by developing the theory of free dg-Lie algebroids initiated in \cite{kapranov_free_2007} in a systematic way. In particular, we describe explicitly their enveloping and jet algebras, which is new up to the author's knowledge. The strategy of studying free Lie algebroids in order to understand the first order approximation of a subscheme is totally in accordance with derived deformation theory and has been highlighted in many recent contributions (\cite{2018arXiv180209556C}, \cite{MR3701353}, \cite{nuiten_homotopical_2019}). It represents a geometric extension of Lurie's theory for formal moduli problems over a field \cite{DAG-X}, \cite{MR3666027} or over a ring \cite{MR3836141}.
\par \smallskip
\textbf{Acknowledgments} The author would like to express his thanks to Eduard Balzin, Damien Calaque and Bernhard Keller for many discussions related to the topic of this paper, as well as the referee whose work led to a substantial improvement of the manuscript.
\section{Recollections}
\subsection{DG-modules}
Let $\mathbf{k}$ be a field of characteristic zero, and let $A$ be a commutative differential graded algebra over $\mathbf{k}$. We will always assume that $A$ is cohomologically concentrated in nonpositive degrees. If nothing is specified, $A$-module will mean \textit{left} differential graded $A$-module.
\par \medskip
The category of unbounded dg $A$-modules is a closed monoidal abelian category $\mathrm{C}(A)$. Its derived category is denoted by $\mathrm{D}(A)$. For general properties of $\mathrm{D}(A)$, we refer to \cite[\S 3]{keller_differential_2006}. An $A$-module $M$ is called \textit{h-flat} if the functor $M \otimes_A \star$ is exact and preserves acyclicity (and thus quasi-isomorphisms).
\par \medskip 
An $A$-module is (strictly) perfect if it is in the smallest subcategory of $\mathrm{C}(A)$ that contains $A$ and is stable by shift (in both directions), extensions, and direct summands. 
In this paper, perfect will always mean strictly perfect. Perfect $dg$-modules have very nice properties: they are dualizable in the monoidal category $\mathrm{C}(A)$, h-flat, and stable under extensions.
\par \medskip
An $A$-module $V$ is reflexive \footnote{Some authors adopt the terminology \textit{weakly dualizable}.} if the natural map from $V$ to its bidual $V^{**}$ is an isomorphism.
\subsection{Graded derivations and DG-Lie algebroids}
\par \medskip
We denote by $\mathbb{T}_A$ the $A$-module of graded derivations of $A$: a homogeneous element $D$ in $\mathbb{T}_A$ satisfies
\[
D(a a')= D(a) a' + (-1)^{|D|. |a|} a D(a').
\] 
We also introduce the right $A$-module of K\"{a}hler differentials of $A$. It is generated over $A$ by the symbols $da$, with the relations given by 
\[
d(aa')=da . a' + (-1)^{|a|.|a'|} da' . a.
\]
It is well-known that $\mathbb{T}_A$ endowed with the graded commutator is a dg-Lie algebra over $\mathbf{k}$. Besides, $\mathbb{T}_A$ is the dual of $\Omega^1_A$. We will also consider $\Omega^1_A$ as a left $A$-module, by putting
$a \omega= (-1)^{|a|.|\omega|} \omega a$. Seeing $\Omega^1_A$ as a bimodule (\textit{see} \S \ref{bimod}) is nice because of the relation
\[
d(aa')=da.a' + a. da'.
\]
\begin{definition}
A $(\mathbf{k}, A)$ dg-Lie algebroid\footnote{The original algebraic definition when $A$ is a usual commutative $\mathbf{k}$-algebra goes back to \cite{rinehart_differential_1963}.} is a pair $(L, \rho)$ where
\begin{enumerate}
\item[--] $L$ is a dg-Lie algebra over the ground field $\mathbf{k}$, endowed with an $A$-module structure.
\item[--] The morphism $\rho \colon L \rightarrow \mathbb{T}_A$, called the anchor map, is an $A$-linear map of dg-Lie algebras.
\item[--] For any homogeneous elements $a$ in $A$ and $\ell_1$, $\ell_2$ in $L$, 
\[
[\ell_1, a \ell_2]-(-1)^{|a| . |\ell_1|}a \,[\ell_1, \ell_2]=\rho(\ell_1)(a)\, \ell_2
\]
\end{enumerate}
\end{definition} 
Note that the two structures (the $\mathbf{k}$ dg-Lie structure and the $A$-module structure) are a priori non compatible, but their lack of compatibility is controlled via the anchor map by the third axiom.
\par \medskip
We denote by $\mathbf{Lie}_{\mathbf{k}/A}$ the category of $(\mathbf{k}, A)$ dg-Lie algebroids. 
\par \medskip
There is a natural forgetful functor from $\mathbf{Lie}_{\mathbf{k}/A}$ to $\mathbf{mod}_{A/\mathbb{T}_A}$ that forgets the dg-Lie structure, and keeps track only on the anchor map. This functor admits a left adjoint
\[
\textit{free} \colon \mathbf{mod}_{A/\mathbb{T}_A} \rightarrow \mathbf{Lie}_{\mathbf{k}/A}.
\]
For any anchored $A$-module $V$, the dg-Lie algebroid $\textit{free}\,(V)$ can be constructed from the free Lie algebra generated by $V$ modulo suitable $A$-linearity relations, we refer to \cite[\S 2.1]{kapranov_free_2007} for the explicit construction.
\subsection{Universal enveloping algebra of a dg-Lie algebroid} \label{toolbox}
For this section, our main references are \cite[\S 2]{calaque_pbw_2014}, \cite[\S 1, 2]{kapranov_free_2007} and \cite{rinehart_differential_1963}. However, we provide many details since this material is not classical.
\par \medskip 
Let us now recall the construction of the universal enveloping algebra of a dg-Lie algebroid. We first define it by hand, and then gives the universal property it satisfies.
\par \medskip
If $L$ is a $(\mathbf{k}, A)$ dg-Lie algebroid, we can endow $\widehat{L}=A \oplus L$ with a structure of a dg-Lie algebra over $\mathbf{k}$ by putting
\[
[a+ \ell, a'+ \ell'] = \rho(\ell)(a')-(-1)^{|a| . |\ell'|} \rho(\ell')(a)+ [\ell, \ell'].
\]
We denote by $U(\widehat{L})$ the enveloping algebra of $\widehat{L}$, and by $U^+(\widehat{L})$ the augmentation ideal of $U(\widetilde{L})$.
\begin{definition}
If $L$ is a $(\mathbf{k}, A)$ dg-Lie algebroid, the universal enveloping algebra $\mathrm{U}_{\mathbf{k}/A}(L)$ of $L$ is the quotient $U^+(\widehat{L})/P$ where $P$ is the two-sided ideal generated by elements of the form $a.\widehat{\ell}- a\widehat{\ell}$ for $a$ in $A$ and $\widehat{\ell}$ in $\widehat{L}$.  It carries a filtration induced by the natural grading of $\mathrm{U}^{+}(\widehat{L})$.
\end{definition}

We now give a more intrinsic characterization of the universal enveloping algebra of a Lie algebroid using universal properties.

\begin{definition}
The category $\mathbf{alg}_{A/ \mathbf{End}_\mathbf{k}(A)}$ is defined as follows:
\begin{enumerate}
\item[--] Objects of the category $\mathbf{alg}_{A/ \mathbf{End}_\mathbf{k}(A)}$ are pairs $(R, \sigma)$ such that $R$ is a unital $A$-algebra, and $\sigma \colon R \rightarrow \mathbf{End}_{\mathbf{k}}(A)$ is a $A$-linear algebra morphism (called anchor) 
\item[--] Morphisms are morphisms of $A$-algebras commuting with the anchors.
\end{enumerate}
\end{definition}

\begin{lemma}
Let $L$ be a Lie algebroid. Then $\mathrm{U}^+_{\mathbf{k}/A}(L)$ is naturally an object of the category $\mathbf{alg}_{A/ \mathbf{End}_\mathbf{k}(A)}$.
\end{lemma}

\begin{proof}
We define an action of $\widehat{L}$ on $A$ by the formula 
\[
(a+\ell). a' = aa' + \rho(\ell)(a'). 
\]
First we check that the action is a Lie action.
\begin{align*}
&(a_1+ \ell_1). [(a_2+\ell_2). a'] 
= (a_1+ \ell_1).(a_2a' +\rho(\ell_2)(a')) \\
&= a_1 a_2 a' + a_1 \rho(\ell_2)(a') + \rho(\ell_1)(a_2a')+\rho(\ell_1)\rho(\ell_2)(a') \\
&= a_1 a_2 a' + a_1 \rho(\ell_2)(a') + (-1)^{|a_2| . |\ell_1|} a_2 \rho(\ell_1)(a') + \rho(\ell_1)(a_2)a' \\
& \qquad \qquad + \rho(\ell_1) \rho(\ell_2)(a')
\end{align*}
Hence we get: 
\begin{align*}
a_1.(a_2.a')-(-1)^{|a_1|. |a_2|}a_2.(a_1.a')=0=[a_1, a_2].a'
\end{align*}
\begin{align*}
a_1.(\ell_2& .a')-(-1)^{|a_1|. |\ell_2|} \ell_2.(a_1.a') \\
&=a_1 \rho(\ell_2)(a')- (-1)^{|a_1|. |\ell_2|} ((-1)^{|a_1| . |\ell_2|} a_1 \rho(\ell_2)(a') + \rho(\ell_2)(a_1)a')  \\
&= - (-1)^{|a_1| . |\ell_2|} \rho(\ell_2)(a_1) = [a_1, \ell_2]. a'
\end{align*}
\begin{align*}
\ell_1.(a_2& .a')-(-1)^{|\ell_1|. |a_2|} a_2.(\ell_1.a') \\
&=(-1)^{|a_2| . |\ell_1|} a_2 \rho(\ell_1)(a') + \rho(\ell_1)(a_2)a'- (-1)^{|\ell_1|. |a_2|} a_2\rho(\ell_1)(a') \\
&=\rho(\ell_1)(a_2)a' = [\ell_1, a_2].a'
\end{align*}
\begin{align*}
\ell_1.(\ell_2& .a')-(-1)^{|\ell_1|. |\ell_2|} \ell_2.(\ell_1.a') \\
&=\rho(\ell_1) \rho(\ell_2)(a')-(-1)^{|\ell_1|. |\ell_2|}\rho(\ell_2)\rho(\ell_1)(a') =\rho([\ell_1, \ell_2])(a')
\end{align*}
Hence there is an induced algebra morphism $\sigma \colon \mathrm{U}^+(\widehat{L}) \rightarrow \mathbf{End}_{\mathbf{k}}(A)$.
Next we prove that $\sigma$ factors through the ideal $P$.
\begin{align*}
\sigma(a.\ell)(a')-\sigma(a \ell)(a') &= \sigma(a) \sigma(\ell)(a') - \sigma (a \ell)(a') \\
&= a. (\ell.a')-(a \ell).a' = a \rho(\ell)(a')-\rho(a \ell)(a')\\
&=0.
\end{align*} 
This yields an algebra morphism $\sigma \colon \mathrm{U}_{\textbf{k}/A}(L) \rightarrow \mathbf{End}_{\mathbf{k}}(A)$.
\end{proof}
We can see the anchor $\sigma$ in a slightly different way. Recall the following 
definition:
\begin{definition}
The algebra $\mathbf{Diff}_{\mathbf{k}}(A)$ of differential operators of $A$ is defined by $\mathbf{Diff}_{\mathbf{k}}(A)=\mathrm{U}_{\mathbf{k}/A}(\mathbb{T}_A)$.
\end{definition}
Any $(\mathbf{k}, A)$ dg-Lie algebroid defines a Lie algebroid map $\rho: L \rightarrow \mathbb{T}_A$. Then the map $\sigma$ is given by the composition
\[
\mathrm{U}(L) \rightarrow \mathrm{U}(\mathbb{T}_A) \simeq \mathbf{Diff}_{\mathbf{k}}(A) \rightarrow \mathbf{End}_{\mathbf{k}}(A).
\]

\begin{definition} \label{fix}
If $R$ is an object of $\mathbf{alg}_{A/ \mathbf{End}_\mathbf{k}(A)}$, we the set $\mathcal{P}(R)$ of primitive elements of $R$ by
\[
\mathcal{P}(R)=\{ r \in R, ra-(-1)^{|a| . |r|}ar= \sigma(r)(a) \}.
\]
\end{definition}

\begin{lemma}
For any object $R$ of  $\mathbf{alg}_{A/ \mathbf{End}_\mathbf{k}(A)}$, the $A$-module $\mathcal{P}(R)$ is naturally a Lie algebroid, the bracket being given by
\[
[r_1, r_2]=r_1 r_2-(-1)^{|r_1| . |r_2|} r_2 r_1.
\]
and the anchor being the restriction of $\sigma$.
\end{lemma}

\begin{proof}
We proceed in several steps.
\par \medskip
\fbox{\textit{$\mathcal{P}(R)$ is an $A$-submodule of $R$}}
\par \medskip
If $r$ is in $\mathcal{P}(R)$, and $a, a'$ in $A$,
\begin{align*}
(a'r)a& - (-1)^{|a|. (|a'|+|r|)}a(a'r) \\
&= (-1)^{|a| . |r|} a'ar+ a' \sigma(r)(a)- (-1)^{|a| . (|a'|+|r|)} aa'r \\
&= \sigma(a'r)(a).
\end{align*}
\par \medskip
\fbox{\textit{$\mathcal{P}(R)$ is stable under the Lie bracket}}
\begin{align*}
[r_1, &r_2]a - (-1)^{|a| . (|r_1|+|r_2|)} a[r_1, r_2] \\
&= r_1 r_2a-(-1)^{|r_1| . |r_2|} r_2 r_1a -  (-1)^{|a| . (|r_1|+|r_2|)}  a r_1 r_2 \\
& \qquad + (-1)^{|a| . (|r_1|+|r_2|) + |r_1| . |r_2|} a r_2 r_1 \\
&= (-1)^{|a| . |r_2|} r_1 a r_2 + r_1 \sigma(r_2)(a) \\
&\qquad -  (-1)^{|r_1|. (|r_2|+|a|)} r_2 a r_1 - (-1)^{|r_1| . |r_2|} r_2 \sigma(r_1)(a)  \\
& \qquad - (-1)^{|a| . |r_2|} r_1a r_2 +  (-1)^{|a| . |r_2|} \sigma(r_1)(a) r_2 \\
& \qquad + (-1)^{|r_1|. (|r_2|+|a|)} r_2 a r_1 -  (-1)^{|r_1|. (|r_2|+|a|)} \sigma(r_2)(a) r_1\\
&= r_1 \sigma(r_2)(a) -  (-1)^{|r_1|. (|r_2|+|a|)} \sigma(r_2)(a)  r_1 \\
&\qquad - (-1)^{|r_1| . |r_2|} \left(r_2 \sigma(r_1)(a)-(-1)^{(|a|+|r_1|) . |r_2|} \sigma(r_1)(a)r_2\right) \\
&=0.
\end{align*}
\fbox{\textit{The restriction of $\sigma$ to $\mathcal{P}(R)$ is a derivation}}
\par \medskip
If $r$ is in $\mathcal{P}(R)$, 
\begin{align*}
\sigma(r)(aa') &= raa'-(-1)^{(|a|+|a'|) . |r|} aa'r \\
&= (ra-(-1)^{|a| . |r|}ar)a'+(-1)^{|a| . |r|} a(ra'-(-1)^{|a'|. |r|} a'r) \\
&= \sigma(r)(a)a'+(-1)^{|a| . |r|} a \sigma(r)(a').
\end{align*}
\fbox{\textit{Defect of $A$-linearity of the bracket}}
\begin{align*}
[r_1, a& r_2]-(-1)^{|a| . |r_1|}a \,[r_1, r_2]\\
&=r_1 a r_2 - (-1)^{|r_1| . (|a|+|r_2|)} ar_2 r_1 - (-1)^{|a| . |r_1|} a r_1 r_2 \\
& \qquad+ (-1)^{(|a|+|r_2|) . |r_1|} ar_2 r_1 \\
&= \sigma(r_1)(a) r_2.
\end{align*}
\end{proof}

\begin{remark}
In \cite[Remark 2.2]{calaque_pbw_2014}, the module of primitive elements of is defined by $\mathcal{P}(R)=R \times_{\mathbf{End}_{\mathbf{k}}(A)} \mathbf{Der}_{\mathbf{k}}(A)$. This definition is not adequate since $\mathcal{P}(R)$ is not always a Lie algebroid. 
\end{remark}

\begin{proposition} \label{ciment}
There is an adjunction
\[
\mathrm{U}_{\mathbf{k}/A} \colon \mathbf{Lie}_{\mathbf{k}/A} \,\begin{matrix}\longrightarrow \\[-0.3cm] \longleftarrow\end{matrix}\,
\mathbf{alg}_{A/ \mathbf{End}_\mathbf{k}(A)} \colon \mathcal{P}
\]
\end{proposition}
\begin{proof}
Let $L$ be a Lie algebroid, $R$ an object of $\mathbf{alg}_{A/ \mathbf{End}_\mathbf{k}(A)}$, and assume to be given a morphism $\phi$ of Lie algebroids from $L$ to $\mathcal{P}(R)$. We can consider the composite morphism
\[
L \rightarrow \mathcal{P}(R) \hookrightarrow R
\]
Since $\mathcal{P}(R)$ contains $A$, we have a morphism
$
\widehat{L} \rightarrow R.
$
We claim that it is a morphism Lie algebras over $\mathbf{k}$ (if we endow $R$ with the bracket given by the commutator). Indeed, we have 
\begin{align*}
&\phi(a_1 + \ell_1) \phi(a_2 + \ell_2)- (-1)^{|a_1| . |a_2|}
\phi(a_2) \phi(a_1)-(-1)^{|a_1| . |\ell_2|}
\phi(\ell_2) \phi(a_1)\\
&\qquad
-(-1)^{|\ell_1| . |a_2|}
\phi(\ell_1) \phi(a_2)-(-1)^{|\ell_1| . |\ell_2|}
\phi(\ell_2) \phi(\ell_1) \\
&= a_1 \phi(\ell_2) -(-1)^{|a_1| . |\ell_2|} \phi(\ell_2) a_1+ \phi(\ell_1) a_2 - (-1)^{|\ell_1| . |\ell_2|}a_2 \phi(\ell_1) \\
& \qquad + \phi(\ell_1 \ell_2)-(-1)^{|\ell_1| . |\ell_2|}\phi(\ell_2)\phi(\ell_1) \\
&=- (-1)^{|a_1| . |\ell_2|} \sigma(\phi(\ell_2)) (a_1) + \sigma(\phi(\ell_1))(a_2) + \phi([\ell_1, \ell_2]) \\
&=- (-1)^{|a_1| . |\ell_2|} \rho(\ell_2)(a_1) + \rho(\ell_1)(a_2) + \phi([\ell_1, \ell_2]) \\
&=\phi(- (-1)^{|a_1| . |\ell_2|} \rho(\ell_2)(a_1) + \rho(\ell_1)(a_2) +[\ell_1, \ell_2] \\
&= \phi([a_1+\ell_1, a_2+\ell_2]).
\end{align*}
Hence we get a morphism of algebras $\overline{\phi} \colon \mathrm{U}^+(\widehat{L}) \rightarrow R$. It is immediate to check that this morphisms factors through $P$. Hence it induces a morphism $\varphi \colon \mathrm{U}_{\mathbf{k}/A}(L) \rightarrow R$.
\par \medskip
Conversely, assume to be given an algebra morphism $f \colon \mathrm{U}_{\mathbf{k}/A}(L) \rightarrow R$. We consider the composition
\[
L \rightarrow \mathrm{U}_{\mathbf{k}/A}(L) \xrightarrow{f} R
\]
Let us proves that it factors through $\mathcal{P}(R)$. We compute: 
\begin{align*}
f(\ell)a-&(-1)^{|a| . |\ell|}af(\ell)=f(\ell)f(a)-(-1)^{|a| . |\ell|}f(a)f(\ell) \\
&=f(\ell.a-(-1)^{|a| . |\ell|}a.\ell) = f(\rho(\ell)(a)) = \rho(\ell)(a) = \sigma(\ell)(a).
\end{align*}
It is straightforward that these construction are mutually inverse, and give an isomorphism
\[
\mathrm{Hom}_{\mathbf{Lie}_{\mathbf{k}/A}}(L, \mathcal{P}(R)) \simeq \mathrm{Hom}_{\mathbf{alg}_{A/ \mathbf{End}_\mathbf{k}(A)}} (\mathrm{U}_{\mathbf{k}/A}(L), R).
\]
\end{proof}
\subsection{Coproduct and jets}

The universal enveloping algebra $\mathrm{U}_{\mathbf{k}/A}(L)$ carries a cocommutative coproduct. We follow the construction given in \cite[\S 2.1]{calaque_pbw_2014} but provide a slightly more conceptual framework.
\par \medskip
Let us introduce some notation: 
\begin{enumerate}
\item[--] $I$ is the right ideal of $\mathrm{U}_{\mathbf{k}/A}(L) \otimes_{\mathbf{k}} \mathrm{U}_{\mathbf{k}/A}(L)$ generated by elements $a \otimes 1 - 1 \otimes a$ for $a$ in $A$.
\item[--] $N$ is the normalizer of the right ideal $I$ of $\mathrm{U}_{\mathbf{k}/A}(L) \otimes_{\mathbf{k}} \mathrm{U}_{\mathbf{k}/A}(L)$.
\item[--] $\widetilde{R}$ is the sub A-module of $\mathrm{U}_{\mathbf{k}/A}(L) \otimes_{\mathbf{k}} \mathrm{U}_{\mathbf{k}/A}(L)$ generated by elements $a \otimes 1$ and $\ell \otimes 1 + 1 \otimes \ell$ for $a$ in $A$ and $\ell$ in $L$.
\item[--] $R_L$ is the image if $\widetilde{R}_L$ in $\mathrm{U}_{\mathbf{k}/A}(L) \otimes_{A} \mathrm{U}_{\mathbf{k}/A}(L)$, that is $R_L=\widetilde{R}_L/I$.
\item[--] $\pi \colon \mathrm{U}_{\mathbf{k}/A}(L) \otimes_{\mathbf{k}} \mathrm{U}_{\mathbf{k}/A}(L) \rightarrow \mathrm{U}_{\mathbf{k}/A}(L) \otimes_{\mathbf{k}} A$ is the map $\mathrm{id} \otimes_{\mathbf{k}} \sigma$, it is a morphism of $\mathbf{k}$-algebras.
\item[--] $\tau \colon \mathrm{U}_{\mathbf{k}/A}(L) \otimes_{\mathbf{k}} \mathrm{U}_{\mathbf{k}/A}(L) \rightarrow \mathbf{End}_{\mathbf{k}}(A)$ is the $A$-linear morphism \footnote{Notice that $\tau$ is not an algebra morphism.} defined by $\tau(x \otimes y)(a)=\sigma(x)(a) \sigma(y)(1).$
\end{enumerate}
Then we have the following result: 
\begin{lemma}
Given a $(\mathbf{k}, A)$-algebroid $L$, the following properties hold: 
\begin{enumerate}
\item[(i)] $\widetilde{R}_L$ is included in $N$, and $R_L$ carries a natural $A$-algebra structure.
\item[(ii)] $\pi(\widetilde{R}_L \cap I)=\{0\}$.
\item[(iii)] The restriction of $\tau$ to $\widetilde{R}_L$ factors through $R_L$ and induces an $A$-linear algebra morphism (that will still be denoted by $\tau$) from $R_L$ to $\mathbf{End}_{\mathbf{k}}(A)$.
\end{enumerate}
\end{lemma}
\begin{proof}
(i) We compute
\begin{align*}
(\ell \otimes 1& + 1 \otimes \ell). (a \otimes 1-1 \otimes a) \\
&= \ell a \otimes 1- \ell \otimes a + (-1)^{|a| . |\ell|} a \otimes \ell  - 1 \otimes \ell a \\
&=(-1)^{|a| . |\ell|} (a \ell \otimes 1 - 1 \otimes a \ell) + (\rho(\ell)(a) \otimes 1-1 \otimes \rho(\ell)(a)) \\
& \qquad+((-1)^{|a| . |\ell|} a \otimes \ell-\ell \otimes a ) \\
&=(-1)^{|a| . |\ell|}  (a \otimes 1-1 \otimes a). (\ell \otimes 1+1 \otimes \ell) \\
& \qquad +(\rho(\ell)(a) \otimes 1-1 \otimes \rho(\ell)(a)) \\
& \in I.
\end{align*}
In the same way, 
\begin{align*}
a \otimes 1. (&a' \otimes 1-1 \otimes a') \\
&= aa' \otimes 1 -a \otimes a' = aa' \otimes 1 -1 \otimes aa' + 1 \otimes aa'-a \otimes a' \\
&= (aa' \otimes 1-1 \otimes aa')+(1 \otimes a-a \otimes 1). 1 \otimes a'.
\end{align*}
Since $I$ is the right ideal, we get the first point.
\par \medskip
(ii) The algebra $\pi(R_L)$ lies inside the subalgebra $\mathrm{U}_{\mathbf{k}/A}$ of $\mathrm{U}_{\mathbf{k}/A} \otimes_{\mathbf{k}} A$. We have a diagram
\[
\xymatrix{ \mathrm{U}_{\mathbf{k}/A}(L) \ar@{^{(}->}[r] \ar[rd]^-{\sim} & \mathrm{U}_{\mathbf{k}/A}(L) \otimes_{\mathbf{k}} A \ar[d] \\
& \mathrm{U}_{\mathbf{k}/A}(L) \otimes_A A
}
\]
Since $\pi(I)$ lies in the kernel of the vertical arrow, we get 
\[
\mathrm{U}_{\mathbf{k}/A} \cap \pi(I)=\{0\}.
\]
(iii) Thanks to (ii), there is an algebra morphism from $R_L$ to $\mathrm{U}_{\mathbf{k}/A}(L)$ making the diagram below commutative: 
\[
\xymatrix{
\widetilde{R}_L \, \ar[d] \ar@{^{(}->}[r] &\mathrm{U}_{\mathbf{k}/A}(L) \otimes_{\mathbf{k}} \mathrm{U}_{\mathbf{k}/A}(L) \ar@/^2pc/[rr]^-{\tau} \ar[r]^-{\pi} \ar[d] & \mathrm{U}_{\mathbf{k}/A}(L) \otimes_{\mathbf{k}} A \ar[r] & \mathbf{End}_{\mathbf{k}}(A) \\
R_L\, \ar@{-->}@/_2pc/[rr] \ar@{^{(}->}[r] & \mathrm{U}_{\mathbf{k}/A}(L) \otimes_{A} \mathrm{U}_{\mathbf{k}/A}(L)  & \mathrm{U}_{\mathbf{k}/A}(L) \ar@{^{(}->}[u] \ar[ru]_-{\sigma} &  
}
\]
\par \bigskip
\par \bigskip
This finishes the proof since $\sigma$ is an $A$-linear algebra morphism.
\end{proof}
The previous lemma allows to consider the algebra $R_L$ as an object of the category $\mathbf{alg}_{A/ \mathbf{End}_\mathbf{k}(A)}$.
\begin{lemma}
The map from $L$ to $\mathrm{U}_{\mathbf{k}/A}(L) \otimes_{A} \mathrm{U}_{\mathbf{k}/A}(L)$ given by the formula $\ell \mapsto \ell \otimes 1 + 1 \otimes \ell$ factors through a Lie algebroid morphism with values in the primitive elements $\mathcal{P}(R_L)$ of $R_L$.
\end{lemma}

\begin{proof}
The element $\ell \otimes 1 + 1 \otimes \ell$ is primitive because
\begin{align*}
(\ell \otimes 1 & + 1 \otimes \ell) . a \otimes 1 - (-1)^{|a| . |\ell|} a \otimes 1 . (\ell \otimes 1 +1 \otimes \ell) \\
&= (\ell a -(-1)^{|a| . |\ell|} a\ell ) \otimes 1 
= \rho(\ell)(a) \otimes 1 
= \tau (\ell \otimes 1 + 1 \otimes \ell)(a)
\end{align*}
We check that the morphism $\ell \rightarrow \ell \otimes 1 + 1 \otimes \ell$ is a Lie algebroid morphism:
\begin{align*}
&(\ell_1 \otimes 1  + 1 \otimes \ell_1).(\ell_2 \otimes 1  + 1 \otimes \ell_2) \\
& \qquad \qquad -(-1)^{|\ell_1| . |\ell_2|} (\ell_2 \otimes 1  + 1 \otimes \ell_2).(\ell_1 \otimes 1  + 1 \otimes \ell_1) \\
&= (\ell_1 \ell_2 \otimes 1 - (-1)^{|\ell_1| . |\ell_2|} \ell_2 \ell_1)  \otimes 1 + 1 \otimes (\ell_1 \ell_2 \otimes 1 - (-1)^{|\ell_1| . |\ell_2|} \ell_2 \ell_1)\\
&= [\ell_1, \ell_2] \otimes 1+ 1 \otimes [\ell_1, \ell_2].
\end{align*}
\end{proof}
As a corollary, using Proposition \ref{ciment}, the map $\ell \mapsto \ell \otimes 1 + 1 \otimes \ell$ from $L$ to $\mathcal{P}(L_R)$ extends uniquely to a morphism $\mathrm{U}_{\mathbf{k}/A}(L) \rightarrow R_L$ in $\mathbf{alg}_{A/\mathbf{End}_{\mathbf{k}}(A)}$. In particular we have a morphism
\[
\Delta \colon \mathrm{U}_{\mathbf{k}/A}(L) \rightarrow \mathrm{U}_{\mathbf{k}/A}(L) \otimes_A \mathrm{U}_{\mathbf{k}/A}(L)
\]
and it is a straightforward calculation to check that it endows $\mathrm{U}_{\mathbf{k}/A}(L)$ with a cocommutative coproduct in the category of $A$-modules, the counit being given by the map $x \rightarrow \sigma(x)(1)$.
\begin{definition}
The jet algebra $\mathrm{J}_{\mathbf{k}/A}(L)$ of a $(\mathbf{k}, A)$ dg-Lie algebroid is the $A$-module $\mathrm{Hom}_{A}(\mathrm{U}_{\mathbf{k}/A}(L), A)$.
\end{definition}
The dual (as left $A$-modules) of the coproduct on $\mathrm{U}_{\mathbf{k}/A}(L)$ endows $\mathrm{J}_{\mathbf{k}/A}(L)$ with an algebra structure that turns it into a cdga in the category of left $A$-modules.
%
%

\subsection{Generalities on bimodules} \label{bimod} $ $
\par \medskip
Let $\textbf{corr}_A$ denote the category of $(A, A)$-bimodules, it is a monoidal (but not symmetric!) category whose unit element is the algebra $A$, considered as a bimodule over itself. If we forget the monoidal structure, it is the category of dg-modules over $A \otimes_{\mathbf{k}} A^{\mathrm{op}}$, so it is abelian.
\par \medskip
The category $\textbf{corr}_A$ has a natural duality functor, that we denote by $\mathbb{D}$, which is defined as follows: for any bimodule $K$, $\mathbb{D}(K)$ is the dual of $K$ considered as a left $A$-module, that is applications $\phi \colon K \rightarrow A$ such that $\phi(ak)= a \phi(k)$. The left and right actions are defined by
\[
\begin{cases}
(a \star \phi) (k) = \phi(ka) \\
(\phi \star a) (k) = \phi(k) a
\end{cases}
\]
Any object $K$ of $\textbf{corr}_A$ defines an endofunctor of $\mathbf{mod}_A$ defined by 
\[
M \rightarrow K \otimes_A M
\]
For any objects $K_1$ and $K_2$ of $\textbf{corr}_A$, we have of course
\[
(K_1 \otimes_A K_2) \otimes_AM \simeq K_1 \otimes_A (K_2 \otimes_A M).
\]
Any $A$-module $M$ defines trivially an object of $\textbf{corr}_A$ with the right action given by $ma=(-1)^{|m|.|a|}am$. In this case the associated functor is nothing but the usual tensor product.
\par \medskip
\raisebox{0.08cm}{\danger} If $K$ is an object of $\textbf{corr}_A$ and $M$ is an $A$-module, then $M$ can be considered as a bimodule as we just explained, so $K \otimes_A M$ is also a bimodule. On the other hand, since $K \otimes_A M$ is a left $A$-module, we can also turn it in a bimodule, which is \textit{not} the same as the previous one. This implies that one should be careful in the calculations.
\par \medskip
Any unital $A$-algebra $R$ defines naturally an object of $\textbf{corr}_A$, since there is a natural morphism from $A$ to $R$. Here we do not require that $A$ is central in the algebra. This case is prototypical of universal enveloping algebras of Lie algebroids: if $L$ is in $\mathbf{Lie}_{\mathbf{k}/A}$ the relation
\[
xa-(-1)^{|a| . |x|} ax = \rho(x)(a)
\] 
for $x$ in $L$ shows that $A$ is in general never central in $\mathrm{U}_{\mathbf{k}/A}(L)$, except when the anchor map is zero.
The corresponding forgetful functor 
\[
\mathbf{alg}_A \rightarrow \mathbf{corr}_A
\] 
admits a left adjoint, simply given by the tensor algebra (in bimodules).
\par \medskip
For any bimodule $K$ and anjy $a$ in $A$, we define the endomorphism $\delta_a$ of $K$ by
\[
\delta_a \colon k \mapsto ak - (-1)^{|a|.|k|} ka.
\]
\begin{definition}
An object $K$ of $\mathbf{corr}_A$ is nilpotent if for any $k$ in $K$ there exists an integer $n$ such that for any $a_1, \ldots, a_n$ in $A$, 
\[
\delta_{a_1} \circ \ldots \circ \delta_{a_n}(k)=0. 
\]
If the number $n$ can be chosen independently of $a$, we say that $K$ is nilpotent of index $\leq n$. We denote by $\mathbf{corr}_A^{\, \textrm{nilp}}$ (resp. $\mathbf{corr}_A^{\, \leq n}$) the full subcategory of $\mathbf{corr}_A$ consisting of nilpotent bimodules (resp. of nilpotent bimodules of order $\leq n$). 
\end{definition}
Note that objects of $\mathbf{corr}_A^{\,0}$ are exactly bimodules associated to $A$-modules. In the sequel, we will mainly deal with bimodules or order $\leq 1$.
\par \medskip
The category, $\mathbf{corr}_A$ carries an involution: for any bimodule $K$, $K^{\mathrm{op}}$ is defined as follows: as a $\mathbf{k}$-vector space it is $K$, and the left and right actions are defined by
\[
\begin{cases}
a \heartsuit k = (-1)^{|a|. |k|}\, ka \\
k \heartsuit a = (-1)^{|a|. |k|}\, ak
\end{cases}
\]
A bimodule is called \textit{symmetric} if it is isomorphic to its opposite.
\par \medskip
\raisebox{0.08cm}{\danger} $\mathbb{D}(K)^{\mathrm{op}}$ is not necessarily isomorphic to  $\mathbb{D}(K^{\mathrm{op}})$. Indeed, if we consider only the underlying $\mathbf{k}$-vector spaces, the former is the \textit{left} dual of $K$, and the latter the \textit{right} dual.

\section{Atiyah bimodules}

\subsection{Construction via anchored modules}

One of the most important bimodule is the algebra $\mathbf{Diff}_{\mathbf{k}}(A)$ of differential operators, which is in fact $\mathrm{U}_{\mathbf{k}/A}(\mathbb{T}_A)$. Its filtered pieces are no longer algebras, but they remain bimodules. For the first step of the filtration, $\mathbf{Diff}^{\,\,\leq 1}_{\mathbf{k}}(A)$ admits the following description: the underlying $\mathbf{k}$-vector space can be identified with $
A \oplus \mathbb{T}_A
$, the left action of $A$ is the natural one, and the right action is given by the formula
\[
(a, Z).a'=(aa'+ Za', (-1)^{|Z| . |a'|} a'Z).
\]
We provide for the reader a sanity check on the sign conventions: 
\begin{align*}
((a, Z).a').a''&=(aa'+ Za', (-1)^{|Z| . |a'|} a'Z)a'' \\
&= (aa'a''+ Za' \times a'' + (-1)^{|Z| . |a'|} a'Za'',  \\
& \qquad \qquad \qquad \qquad \qquad \qquad (-1)^{|Z| . |a'| + (|Z|+ |a|') . |a''|} a'' a' Z) \\
&= (aa'a''+Z(a'a''), (-1)^{|Z| . (|a'|+ |a''|)} a' a'' Z) \\
&= (a, Z) . (a'a'').
\end{align*}
There is an exact sequence of bimodules
\[
0 \rightarrow A \rightarrow \mathbf{Diff}^{\,\, \leq 1}_{\mathbf{k}}(A) \rightarrow \mathbb{T}_A \rightarrow 0.
\]
The dual of $\mathbf{Diff}^{\,\,\leq 1}_{\mathbf{k}}(A)$ is $\Omega^1_A \oplus A$, the right action of $A$ being the natural one, and the left action being given by
\[
a'. (\omega, a)=(a' \omega+da'.a, a'a).
\]
The associated endofunctor of $\mathbf{mod}_A$ attaches to any $A$-module its module of $1$-jets. Again, there is an exact sequence of bimodules
\[
0 \rightarrow \Omega^1_A \rightarrow \mathbb{D} (\mathbf{Diff}^{\,\, \leq 1}_{\mathbf{k}}(A)) \rightarrow A \rightarrow 0.
\]
These two examples can be extended using base change to any anchored $A$-module $(V, \rho)$ in $\mathbf{mod}_{A/\mathbb{T}_A}$: we put 
$\Sigma_V=V \times_{\mathbb{T}_A} \mathbf{Diff}^{\,\, \leq 1}_{\mathbf{k}}(A)$. The explicit description of $\Sigma_V$ runs as follows: as a $\mathbf{k}$-vector space it is isomorphic to $A \oplus V$. The left $A$-module structure is the naive one, and the right $A$-module structure is given by the formula
\[
(a, v).a'=(aa'+\rho(v)a', (-1)^{|a'| . |v|} a'v).
\]
 Besides, there is an exact sequence of bimodules
\begin{equation} \label{stuck1}
0 \rightarrow A \rightarrow \Sigma_V \rightarrow V \rightarrow 0.
\end{equation}
The dual of $\Sigma_V$ is $V^* \oplus A$ as $\mathbf{k}$-vector space. The right $A$-module structure is the naive one, and the left $A$-module structure is given by the formula
\[
a'. (\theta, a)=(a' \theta +\rho^*(da') a, a'a).
\]
where $\rho^* \colon \Omega^1_A \rightarrow V^*$ is the transpose of $\rho$. Besides, there is an exact sequence
\begin{equation}  \label{stuck2}
0 \rightarrow V^* \rightarrow \Sigma_V^* \rightarrow A \rightarrow 0.
\end{equation}
Remark that $\Sigma_V$ and $\Sigma_V^*$ lie in $\mathbf{corr}_{A}^{\,\leq 1}$.
\begin{lemma} \label{zwip}
If $V$ is perfect, then so is $\Sigma_V$ considered as left or right $A$-module.
\end{lemma}
\begin{proof}
This follows directly from \eqref{stuck1} and the fact that perfect dg-modules are stable by extension.
\end{proof}

\begin{lemma} \label{swap}
For any anchored $A$-module $V$, the bimodule $\Sigma_V^*$ is symmetric.
\end{lemma}

\begin{proof}
We define $\chi \colon \Sigma_V^* \rightarrow (\Sigma_V^*)^{\mathrm{op}}$ as follows: 
\[
\chi(\theta, a)=(\rho^*(da)-\theta, a)
\]
We check that $\chi$ is a morphism of bimodules: 
\begin{align*}
\chi(a'(\theta, a)) &= \chi(a' \theta + \rho^*(da')a, a'a) \\
&= (\rho^*(d(a'a))-a' \theta - \rho^*(da'.a), a'a) \\
&= (\rho^*(a' da)-a' \theta , a'a) \\
&= ((-1)^{|\theta|.|a'|} \theta a'-(-1)^{|a|.|a'|} \rho^*(da)a', (-1)^{|a|.|a'|} aa') \\
&= (-1)^{|a|.|a'|} (\rho^*(da), a) a' -  (-1)^{|\theta|.|a'|} (\theta, 0) a'  \\
&= a' \heartsuit (\rho^*(da), a)-a' \heartsuit (\theta, 0)  \\
&= a' \heartsuit (\chi(\theta, a))
\end{align*}
and 
\begin{align*}
(\chi(\theta, a)) \heartsuit a' &= (\rho^*(da)-\theta, a) \heartsuit a' \\
&= (-1)^{|a|.|a'|} a'(\rho^*(da), a)-(-1)^{|\theta|.|a'|} a'(\theta, 0)\\ 
&=(-1)^{|a|.|a'|} (a' \rho^*(da) + \rho^*(da')a, a'a)-(-1)^{|\theta|.|a'|} (a' \theta, 0 )\\
&= (\rho^*(d(aa'))-\theta a', aa')\\
&= \chi(\theta a', aa') \\
&= \chi((\theta, a)a').
\end{align*}
Since $\chi$ is an involution, it is an isomorphism.
\end{proof}
The symmetry of the Atiyah bimodules admits the following more intrinsic explanation: if $I$ is the two sided ideal generated by elements $a \otimes 1-1 \otimes a$ in $A \otimes_{\mathbf{k}} A$, then $(A \otimes_{\mathbf{k}} A) / I^2$ is a bimodule, we call it $A^{(1)}$. Then we have an exact sequence
 \[
 0 \rightarrow I / I^2 \rightarrow A^{(1)} \rightarrow A \rightarrow 0
 \]
and $I/I^2$ is isomorphic to $\Omega^1_A$ via the map attaching $da$ to $a \otimes 1-1 \otimes a$. Then we claim the following: 

\begin{lemma}
$\Sigma_V^*$ is isomorphic to the pushout of the diagram $\xymatrix{ \Omega^1_A \ar[r] \ar[d] _-{\rho^*}& A^{(1)} \\V^*& } $.
\end{lemma}

\begin{proof}
Recall that $\mathbf{corr}_A$ is abelian. We have a cartesian diagram 
\[
\xymatrix{\Sigma_V \ar[r] \ar[d] & \mathbf{Diff}_{\mathbf{k}}^{\, \leq 1}(A)\ar[d] \\ V \ar[r]^-{\rho} &\mathbb{T}_A}
\] 
Since the map $\mathbf{Diff}_{\mathbf{k}}^{\, \leq 1}(A) \rightarrow \mathbb{T}_A$ is surjective, the diagram is cocartesian as well. Hence the dual diagram
\[
\xymatrix{\Omega^1_A \ar[r] \ar[d]_-{\rho^*} & \mathbb{D}(\mathbf{Diff}_{\mathbf{k}}^{\, \leq 1}(A))\ar[d] \\ V^* \ar[r] &\Sigma_V^*}
\]
is cartesian. Since the morphism $\Omega^1_A \rightarrow  \mathbb{D}(\mathbf{Diff}_{\mathbf{k}}^{\, \leq 1}(A))$ is injective, this diagram is also cocartesian. 
\par \medskip
The only remaining thing to prove is the identification between the two bimodules $\mathbb{D}(\mathbf{Diff}_{\mathbf{k}}^{\, \leq 1}(A))$ and $A^{(1)}$. We already have an explicit model for $\mathbb{D}(\mathbf{Diff}_{\mathbf{k}}^{\, \leq 1}(A))$. We provide an isomorphism 
\[
\theta \colon A^{(1)} \rightarrow \mathbb{D}(\mathbf{Diff}_{\mathbf{k}}^{\, \leq 1}(A)) 
\]
by putting $\theta(a_1 \otimes a_2)= (da_1 . a_2, a_1a_2)$. This maps obviously respects the right action on $A$, but also the left since
\begin{align*}
\theta(aa_1 \otimes a_2)&=(da. a_1a_2 + a da_1.a_2, aa_1a_2) \\
&= a (da_1 . a_2, a_1a_2).
\end{align*}
It is easy to prove that it is an isomorphism.
\end{proof}
As a corollary, we see directly that $\Sigma_V^*$ is symmetric, since it is a pushout of a diagram of symmetric correspondences.

\subsection{Duality for Atiyah bimodules}
In this section, we prove a technical compatibility result concerning the behaviour of Atiyah bimodules under duality. 

\begin{proposition} \label{salaud}
Let $(V, \rho)$ be an anchored $A$-module, and assume that $V$ is perfect. Then:
\begin{enumerate}
\item The sequence $0 \rightarrow (V \otimes_A M)^* \rightarrow (\Sigma_V \otimes_A M)^* \rightarrow M^* \rightarrow 0$ is exact. \vspace{0.2cm}
\item The dual of $\Sigma_V \otimes_A M$ is canonically isomorphic to $M^* \otimes_A \Sigma_V^*$, and this isomorphism yields an isomorphism of exact sequences
\[
\xymatrix{
0 \ar[r]& M^* \otimes_A V^* \ar[r] \ar[d]^-{\wr} &M^* \otimes_A \Sigma_V^* \ar[r] \ar[d]^-{\wr} & M^* \ar[r] \ar@{=}[d] & 0  \\
0 \ar[r] &  (V \otimes_A M)^* \ar[r] &  (\Sigma_V \otimes_A M)^* \ar[r]& M^* \ar[r]&
0}
\]
\end{enumerate}
\end{proposition}

\begin{proof}
We consider the map $\Theta \colon  M^* \otimes_A \Sigma_V^* \rightarrow (\Sigma_V \otimes_A M)^*$ defined as follows: 
\[
\Theta (\delta \otimes \phi)   (s \otimes m) = (-1)^{|\phi| . (|\delta|+|m|)} \phi(s)\delta(m) +\phi \{\rho(\pi_V(s))(\delta(m))\}.
\]
In order to prove that this formula does make sense, we must check the following list of properties:
\par \medskip
\fbox{$\Theta(\phi \otimes \delta)$ is well defined on $\Sigma_V \otimes_A M$} 
\par \medskip
First we notice that on $A$, viewed as a submodule of $\Sigma_V$, we have 
\[
\phi(aa')=(-1)^{|\phi| . |a'|}\phi(a)a'. 
\]
Then we compute:
\begin{align*}
\Theta (\delta & \otimes \phi)  (sa \otimes m)\\
&= (-1)^{|\phi|. (|\delta|+|m|)} \phi(sa) \delta(m) + \phi \{\rho(\pi_V(sa))(\delta(m))\} \\
&= (-1)^{|\phi| . (|\delta|+|m|)} \phi \{(-1)^{|s| . |a|}as +  \rho(\pi_V(s))(a)\} \delta(m) \\
& \qquad   + (-1)^{|s| . |a|}\phi \{a\rho(\pi_V(s))(\delta(m))\} \\
&= (-1)^{|\phi| . (|\delta|+|m|) + |s| . |a|} a \phi(s) \delta(m)  + \phi \{\rho(\pi_V(s))(a) \delta(m)\} \\
& \qquad  + (-1)^{|s| . |a|}\phi \{a\rho(\pi_V(s))(\delta(m))\} \\
&=(-1)^{|\phi|. (|\delta|+|m|+|a|)} \phi(s) \delta(am) + \phi \{\rho(\pi_V(s))(\delta(am))\} \\
&=\Theta (\delta \otimes \phi) (a \otimes sm).
\end{align*}

\par \medskip
\fbox{\textit{$\Theta(\delta \otimes \phi)$ is $A$-linear}} 
\par \medskip
This is obvious, since both $\pi_V$ and $\rho$ are.
\par \medskip
\fbox{\textit{$\Theta$ is well defined}}
\par \medskip
This is again a nice calculation: 
\begin{align*}
\Theta (\delta a & \otimes \phi)  (s \otimes m)\\
&=(-1)^{|\phi|. (|\delta|+|a|+|m|)} \phi(s) \delta(m) a + \phi \{\rho(\pi_V(s))(\delta(m)a)\} \\
&=(-1)^{(|\phi|+|a|)(|\delta|+|m|)+|s|.|a|} a \phi(s) \delta(m) \\
& \qquad+ (-1)^{(|\delta|+|m|).|a|} \phi \{\rho(\pi_V(s))(a\delta(m))\}\\
&= (-1)^{(|\phi|+|a|)(|\delta|+|m|)}\phi(sa) \delta(m) \\
& \qquad - (-1)^{(|\phi|+|a|)(|\delta|+|m|)} \phi \{\rho(\pi_V(s))(a)\} \delta(m) \\
& \qquad + (-1)^{(|\delta|+|m|).|a|} \phi \{\rho(\pi_V(s))(a\delta(m))\}\\
&=(-1)^{(|\phi|+|a|)(|\delta|+|m|)}\phi(sa) \delta(m) \\
& \qquad  +(-1)^{(|\delta|+|m|).|a|} \phi \{\rho(\pi_V(s))(a\delta(m))-\rho(\pi_V(s))(a) \delta(m) \} \\
&= (-1)^{(|\phi|+|a|)(|\delta|+|m|)}\phi(sa) \delta(m) \\
&\qquad  +(-1)^{(|\delta|+|m|).|a|} \phi \{(-1)^{|s|.|a|} a \rho(\pi_V(s))(\delta(m))\}\\
&= (-1)^{(|\phi|+|a|)(|\delta|+|m|)}\phi(sa) \delta(m) + \phi \{\pi_V(sa)\} \delta(m) \\
&= \Theta (\delta   \otimes a\phi)  (s \otimes m).
\end{align*} 
\par \medskip
\fbox{\textit{$\Theta$ is right $A$-linear}}
\par \medskip
\begin{align*}
(\Theta(\delta \otimes & \phi) a)(s \otimes m) \\
&= \Theta(\delta \otimes \phi)(s \otimes m) a \\
&= (-1)^{|\phi| . (|\delta|+|m|)} \phi(s)\delta(m) a +\phi \{\rho(\pi_V(s))(\delta(m)) \} a\\
&= (-1)^{(|\phi|+|a|).(|\delta|+|m|)} \phi(s) a \delta(m) + \phi \{\rho(\pi_V(s))(\delta(m)) \} a \\
&= \Theta(\delta \otimes \phi a)(s \otimes m).
\end{align*}
\par \medskip
Next, we claim that the diagram below
\[
\xymatrix{
0 \ar[r]& M^* \otimes_A V^* \ar[r] \ar[d]^-{\wr} &M^* \otimes_A \Sigma_V^* \ar[r] \ar[d]^-{\Theta} & M^* \ar[r] \ar@{=}[d] & 0  \\
0 \ar[r] &  (V \otimes_A M)^* \ar[r] &  (\Sigma_V \otimes_A M)^* \ar[r]& M^*
}
\]
commutes. For the left square, if $\phi$ belongs to $V^*$, $\phi$ vanishes on $A$, so
\[
\Theta (\delta \otimes \phi)   (s \otimes m) = (-1)^{|\phi| . (|\delta|+|m|)} \phi(s)\delta(m) .
\]
For the right square, we have $\pi_V(1)=0$ so 
\[
\Theta (\delta \otimes \phi)   (1 \otimes m) = (-1)^{|\phi| . (|\delta|+|m|)} \phi(1)\delta(m)=(\phi(1)\delta))(m).
\]
We can conclude: the bottom sequence is also exact and the middle vertical arrow is an isomorphism.
\end{proof}

\subsection{Derived connections and the HKR class}
Let $V$ be an $A$-module endowed with an anchor $\rho \colon V \rightarrow \mathbb{T}_A$. 
\begin{center}
\textit{From now on, we will always assume that $V$ is perfect.} 
\end{center}
Then, thanks to \eqref{stuck1}, $\Sigma_V$ and $\Sigma_V^*$ are also perfect, when considered as left or right $A$-modules. In particular, all these modules are h-flat.

\begin{definition}
Let $M$ be an $A$-module. A derived $(V, \rho)$ connection on $M$ is a section in $\mathrm{D}(A)$ of the morphism $
\Sigma_V^*  \otimes_A M   \rightarrow M.$
\end{definition} 

\begin{lemma}
A derived $(V, \rho)$ connection on an $A$-module $M$ is given by an equivalence class of triplet $(\widetilde{M}, q, \nabla)$ where
\begin{enumerate}
\item[--] $q \colon \widetilde{M} \rightarrow M$ is a $A$-linear quasi-isomorphism.
\item[--] $\nabla \colon \widetilde{M} \rightarrow V^* \otimes_A M$ is a $\mathbf{k}$-linear morphism such that for all $a$ in $A$ and $m$ in $\widetilde{M}$ respectively, 
\[
\nabla(a\tilde{m})= a \nabla(q(\tilde{m})) +\rho^*(da) \otimes q(\tilde{m})
\]
where $\rho^*$ is the transpose of $\rho$. 
\end{enumerate}
\par \medskip
Besides, two derived connections $\nabla_1$ and $\nabla_2$ are identified if there is a diagram
\[
\xymatrix{ & \widetilde{M}_3 \ar[ld]_-{\sim} \ar[rd]^-{\sim} \ar[dd]^-{\nabla_3}& \\
\widetilde{M}_1 \ar[rd]_-{\nabla_1} && \widetilde{M}_2 \ar[ld]^-{\nabla_2} \\
&V^* \otimes M&
}
\]

\end{lemma}

\begin{proof}
A derived connection can be represented by a span
\[
\xymatrix{
\widetilde{M} \ar[d]_-{\Phi}^-{\wr} \ar[rd]^-{q} & \\
M  &  \ar[l] \Sigma_V^*  \otimes_A M
}
\]
where $\widetilde{M}$ is a cofibrant\footnote{Here we use the projective model structure on $A$-modules.} replacement of $M$ (since all $A$-modules are fibrant). As a $\mathbf{k}$-vector space, $\Sigma_V^*  \otimes_A M$ is isomorphic to $M \oplus (V^* \otimes_A M)$ so $\Phi$ can be written as $(q, \nabla)$. The condition that $u$ is $A$-linear directly  gives Leibniz rule for $\nabla$.
\end{proof}

Lastly we give the construction of the dual connection: 

\begin{lemma} \label{dudu}
If $\nabla$ is a derived connection on a module $M$, it induces a natural connection $\nabla^*$ on $M^*$.
\end{lemma}

\begin{proof}
Let us consider the evaluation morphisms 
\[
\begin{cases}
\mathbf{ev} \colon V^* \otimes_A M \otimes_A M^* \rightarrow V^* \\
\mathbf{ev}^* \colon  \widetilde{M} \otimes_A V^* \otimes_A \widetilde{M}^*  \rightarrow V^*
\end{cases}
\]
given by
\[
\begin{cases}
\mathbf{ev}(\delta \otimes m \otimes \phi)= \delta \otimes \phi(m) \\
\mathbf{ev}^*(\tilde{m} \otimes \delta \otimes \varphi)= (-1)^{|\delta|. |\varphi|}  \varphi(\tilde{m}) \delta.
\end{cases}
\]
Then we define
\[
\nabla^* \colon M^* \rightarrow V^* \otimes \widetilde{M}^*
\] 
by the formula
\[
\mathbf{ev}^*(\tilde{m} \otimes \nabla^*\phi)=\rho^*( d\phi(m))- \mathbf{ev}(\nabla (\tilde{m})\otimes \phi)
\]
where we put $q(\tilde{m})=m$ to lighten notation. 
First we must verify that this formula makes sense, that is, is $A$-linear in $m$. 
\begin{align*}
\mathbf{ev}^*&(a\tilde{m} \otimes \nabla^* \phi)\\
&=\rho^*( d\phi(am))- \mathbf{ev}(\nabla (a\tilde{m})\otimes \phi) \\
&= a \rho^*( d\phi(m)) + \rho^*(da) \phi(m) - \mathbf{ev}(a \nabla(\tilde{m}) \otimes \phi + \rho^*(da) \otimes m \otimes \phi) \\
&=a \rho^*( d\phi(m))-a  \mathbf{ev}(\nabla(\tilde{m}) \otimes \phi) \\
&= a\, \mathbf{ev}^*(\tilde{m} \otimes \nabla^*\phi).
\end{align*}

Then we check Leibniz rule: 
\begin{align*}
&\mathbf{ev}^*(\tilde{m} \otimes \nabla^*(a\phi)) \\
& =(-1)^{|\phi|. |a|} \{\rho^*( d(\phi(m)a)) - \mathbf{ev}(\nabla (\tilde{m})\otimes \phi a) \} \\
& = (-1)^{|\phi|. |a|} \phi(m) \rho^*(da) +  (-1)^{|\phi|. |a|}  \{\rho^*( d(\phi(m))) - \mathbf{ev}(\nabla (\tilde{m})\otimes \phi ) \} a \\
&= \mathbf{ev}^*(\tilde{m} \otimes q^* \circ \rho^*(da) \otimes \phi) +   (-1)^{|\phi|. |a|}\mathbf{ev}^*(\tilde{m} \otimes \nabla^* \phi ) a \\
&= \mathbf{ev}^*(\tilde{m} \otimes q^* \circ \rho^*(da) \otimes \phi) + \tilde{m} \otimes a\nabla^* \phi ).
\end{align*}
\end{proof}

Let $(V, \rho)$ be an anchored perfect $A$-module . For any $A$-module $M$, the $A$-module $\Sigma_V \otimes_A M$ fits into an exact sequence
\begin{equation} \label{HKR}
0 \rightarrow M \rightarrow  \Sigma_V  {\otimes}_A M \rightarrow V  {\otimes}_A M \rightarrow 0
\end{equation}
\begin{definition}
The HKR class of $M$ is the morphism 
\[
\Theta_M \colon V[-1]  {\otimes}_A M \rightarrow M
\]
in $\mathrm{D}(A)$ given by the extension class of \eqref{HKR}.
\end{definition}

\begin{proposition} \label{stargarder}
The HKR class of an $A$-module $M$ vanishes if and only if $M$ admits a derived $(V, \rho)$-connection.
\end{proposition}
\begin{proof}
We will give two different proofs of this result, which are somehow dual. However, the second proof works only under the extra assumption that $M$ is reflexive (which will be almost always the case in the applications).
\par \medskip
\textbf{Proof 1}\,\, As $V$ is perfect, we can see a derived $(V, \rho)$ connection on $M$ as an $A$-linear map
\[
\nabla \colon V \otimes_{\mathbf{k}} \widetilde{M} \rightarrow M
\]
such that $
\forall v \in V, \forall a \in A, \forall \tilde{m} \in \widetilde{M}$,
\[ 
\nabla_{v} (a\tilde{m})= (-1)^{|v| |a|} a \nabla_v(\tilde{m}) + \rho(v)(a) q(\tilde{m}).
\]
Assume to be given such a connection. We claim that the map
\[
\Sigma_V \otimes_{A} \widetilde{M} \rightarrow M
\]
given by
\begin{equation} \label{pigeonnier}
\mu \{(a, v) \otimes \tilde{m}\} = a q(\tilde{m}) + \nabla_{v} \tilde{m}
\end{equation}
is well defined. We check: 
\begin{align*}
\mu \{(a, v)a' \otimes \tilde{m}\}  &= \mu \{(aa' + \rho(v)(a'), va') \otimes \tilde{m} \} \\
&= (aa' + \rho(v)(a')) q(\tilde{m})+\nabla_{va'} \tilde{m} \\
&= (aa' + \rho(v)(a')) q(\tilde{m}) + (-1)^{|v| . |a|'} a' \nabla_v \tilde{m} \\
&=a q(a'\tilde{m}) + \nabla_v (a' \tilde{m}) \\
&=  \mu \{(a, v) \otimes a' \tilde{m}\}.
\end{align*}
Now the diagram $\xymatrix{ \widetilde{M} \ar[r] \ar[rd]_-{q} & \Sigma_V \otimes \widetilde{M} \ar[d]^-{\mu} \\
& M }
$ commutes, so $\Theta_M$ vanishes. Conversely, if $\Theta_M$ vanishes, there exists a quasi-isomorphism $q \colon \widetilde{M} \rightarrow {M}$ and a map \[
\delta \colon V \otimes_A \widetilde{M} \rightarrow \Sigma_V \otimes_A M
\] 
such that  the diagram $\xymatrix{\Sigma_V \otimes_A {M} \ar[dr] & V \otimes_A \widetilde{M} \ar[l]_-{\delta} \ar[d]^-{\mathrm{id}_V \otimes q} \\ & V \otimes_A M}$
commutes. 
\par \medskip
Let us now consider the map $\nabla \colon V \otimes_{\mathbf{k}} \widetilde{M} \rightarrow \Sigma_V \otimes_A M$ defined by the formula
\[
\nabla(v \otimes \tilde{m}) = \delta (v \otimes q(m)) - (0, v) \otimes q(m).
\]
The composition 
\[
V \otimes_{\mathbf{k}} \widetilde{M} \xrightarrow{\nabla} \Sigma_V \otimes_A M \rightarrow V \otimes_A M
\] 
vanishes, so $\nabla$ factors through $M$. Since $\nabla$ is obviously $A$-linear, it remains to check that $\nabla$ satisfies Leibniz rule. 
\begin{align*}
\nabla (v \otimes a \tilde{m})&= \delta (v \otimes a q(m)) - (0, v) \otimes aq(m) \\
&= \delta (v a\otimes  q(m)) - (\rho(v)(a), va) \otimes q(m) \\
&= (-1)^{|v|.|a|} a \nabla(v \otimes \tilde{m})  (\rho(v)(a), 0) \otimes q(m).
\end{align*}
\par \medskip
\textbf{Proof 2}\,\, This proof is closer in spirit to the construction of \cite{kapranov_rozansky-witten_1999} but for this approach it is needed that $M$ be reflexive. Assuming this and using Proposition \ref{salaud} (1), the sequence \eqref{HKR} splits if and only if and only if the dual sequence
\[
0 \rightarrow (V \otimes_A M)^* \rightarrow  (\Sigma_V  {\otimes}_A M)^* \rightarrow M^* \rightarrow 0
\]
splits. According to Proposition \ref{salaud} (2), this sequence is isomorphic
\[
0 \rightarrow M^* \otimes_A V^* \rightarrow   M^* \otimes_A \Sigma_V^* \rightarrow M^* \rightarrow 0.
\]
Lastly, using Lemma \ref{swap}, this sequence is isomorphic to 
\[
0 \rightarrow V^* \otimes_A M^* \rightarrow    \Sigma_V^* \otimes_A M^* \rightarrow M^* \rightarrow 0
\]
so it splits if and only if $M^*$ admits a derived $(V, \rho)$ connection. Thanks to Lemma \ref{dudu} and to the fact that $M$ is reflexive, there is a one to one correspondence between $(V, \rho)$ connections on $M$ and on $M^*$. This finishes the proof.
\end{proof}

\begin{proposition} \label{eberswalder}
For any $A$-modules $M_1$ and $M_2$, we have 
\[
\Theta_{M_1 \overset{\mathbb{L}}{\otimes}_A M_2}=\Theta_{M_1} \overset{\mathbb{L}}{\otimes}_A \mathrm{id}_{M_2} +  \mathrm{id}_{M_1} \overset{\mathbb{L}}{\otimes}_A \Theta_{M_2}.
\]
\end{proposition}

\begin{proof}

We can assume that $M_1$ and $M_2$ are h-flat. We introduce an automorphism $s \rightarrow s^{\dagger}$ of $\Sigma_V$ given by the formula $(a, v)^{\dagger}=(\frac{a}{2}, v)$. Let
\[
T=\dfrac{M_1 \otimes_A \Sigma_V \otimes_A M_2 \times_{V \otimes_A M_1 \otimes_A M_2} M_2 \otimes_A \Sigma_V \otimes_A M_1}{\{m_1 \otimes 1 \otimes m_2, - (-1)^{|m_1|.|m_2|} m_2 \otimes 1 \otimes m_1 \}}.
\]
We claim that the map $\tau \colon \Sigma_V \otimes_A M_1 \otimes_A M_2  \xrightarrow{\tau} T$ given by 
\[
\tau(s \otimes m_1 \otimes m_2)=\Bigl( (-1)^{|s|.|m_1|} m_1 \otimes s^{\dagger} \otimes m_2, (-1)^{(|s|+|m_1|) |m_2|} m_2 \otimes s^{\dagger} \otimes m_1 \Bigr)
\]
is well defined. This claim follows from the following calculations, where we put $\vartheta=\frac{\rho(\pi_V(s))(a')}{2}$ to lighten notation:
\par \medskip
\fbox{$\tau( s \otimes m_1 \otimes a' m_2) - \tau (s \otimes m_1 a' \otimes m_2)$} 
\par \medskip
\begin{align*}
\textrm{\ding{62}} \, &\textrm{First component} \\
&(-1)^{|s|.|m_1|} m_1 \otimes s^{\dagger} \otimes a'm_2 - (-1)^{|s|.(|m_1|+|a'|)} m_1 a' \otimes s^{\dagger} \otimes m_2 \\
&=(-1)^{|s|.|m_1|} m_1 \otimes (s^{\dagger} a' - (-1)^{|s|.|a'|} a' s^{\dagger}) m_2 \\
&= 2(-1)^{|s|.|m_1|} m_1 \otimes \vartheta \otimes m_2 \\
&= 2(-1)^{|a'|.|m_1|} \vartheta \,m_1 \otimes 1 \otimes m_2. \\
& \\
\textrm{\ding{62}} \, &\textrm{Second component} \\
& (-1)^{(|s|+|m_1|)(|a'|+|m_2|)} a' m_2 \otimes s^{\dagger} \otimes a_1\\
& \qquad \qquad -(-1)^{(|s|+|m_1|+|a'|) |m_2|} m_2 \otimes s^{\dagger} \otimes m_1 a' \\
&=(-1)^{|s|.|m_2|+|m_1|.|m_2|+|a'|.|m_2|+|a'|.|m_1|} m_2 \otimes ((-1)^{|s|.|a'|}a' s^{\dagger}-s^{\dagger}a') \otimes m_1\\
&=-2(-1)^{|s|.|m_2|+|m_1|.|m_2|+|a'|.|m_2|+|a'|.|m_1|} m_2 \otimes  \vartheta \otimes m_1 \\
&=-2(-1)^{|m_1|.|m_2|+|a'|.|m_1|} \vartheta \,m_2 \otimes 1 \otimes m_1
\end{align*}
\par \medskip
\fbox{$\tau(sa' \otimes m_1 \otimes  m_2) - \tau (s \otimes a'm_1 \otimes m_2)$} 
\par \medskip
First we remark that if $s=(a, v)$,
\begin{align*}
(sa')^{\dagger}&=(aa'+\rho(v)a', va')^{\dagger}=\left(\frac{aa'+\rho(v)a'}{2}, va'\right) \\
&= s^{\dagger} a' - \left(\frac{\rho(v)a'}{2}, 0\right)\\
&= s^{\dagger} a' - \vartheta
\end{align*}
Using this, we compute: 
\par \smallskip
\begin{align*}
\textrm{\ding{62}} \, &\textrm{First component} \\
&(-1)^{(|s|+|a'|)|m_1|} m_1 \otimes (sa')^{\dagger} \otimes m_2 - (-1)^{|s|.(|a'|+|m_1|)} a'm_1 \otimes s \otimes m_2 \\
&=(-1)^{(|s|+|a'|)|m_1|} m_1 \otimes (s^{\dagger} a'-(-1)^{|s|.|a'|} a' s^{\dagger}-\vartheta) \otimes m_2 \\
&=(-1)^{(|s|+|a'|)|m_1|} m_1 \otimes \vartheta \otimes m_2 \\
&= \vartheta m_1 \otimes 1 \otimes m_2. \\
& \\
\textrm{\ding{62}} \, &\textrm{Second component} \\
&(-1)^{(|s|+|a'|+|m_1|) |m_2|} (m_2 \otimes (sa')^{\dagger} \otimes m_1 - m_2 \otimes s^{\dagger} \otimes a'm_1) \\
&=(-1)^{(|s|+|a'|+|m_1|) |m_2|} m_2 \otimes ((sa')^{\dagger} - s^{\dagger} \otimes a') \otimes m_1 \\
&=-(-1)^{(|s|+|a'|+|m_1|) |m_2|} m_2 \otimes \vartheta \otimes m_1 \\
&=-(-1)^{|m_1|.|m_2|} \vartheta m_2 \otimes 1 \otimes m_1.
\end{align*}
Now the Baer sum of the two exact sequences
\[
\begin{cases}
0 \rightarrow M_1 \otimes_A M_2 \rightarrow M_1 \otimes_A \Sigma_V \otimes_A M_2 \rightarrow V \otimes_A M_1 \otimes_A M_2 \rightarrow 0 \\ 
0 \rightarrow M_1 \otimes_A M_2 \rightarrow M_2 \otimes_A \Sigma_V \otimes_A M_1 \rightarrow V \otimes_A M_1 \otimes_A M_2
\rightarrow 0
\end{cases}
\]
is
\[
0 \rightarrow M_1 \otimes_A M_2 \rightarrow  T \rightarrow V \otimes_A M_1 \otimes_A M_2 \rightarrow 0
\]
where the first inclusion is given by $m_1 \otimes m_1 \rightarrow (m_1 \otimes 1 \otimes m_2, 0)$ and the second one is the natural one. We claim we have a commutative diagram
\[
\xymatrix{
0 \ar[r] &  M_1 \otimes_A M_2 \ar[r] \ar@{=}[d] & \Sigma_V \otimes_A M_1 \otimes_A M_2 \ar[r] \ar[d]^-{\tau}  & V \otimes_A M_1 \otimes_A M_2 \ar[r] \ar@{=}[d] & 0 \\
0 \ar[r] &  M_1 \otimes_A M_2 \ar[r] & T \ar[r]  & V \otimes_A M_1 \otimes_A M_2 \ar[r] & 0
}
\]
The commutativity of the right square is straightforward, and the commutativity of the left square follows from
\begin{align*}
&\tau(1 \otimes m_1 \otimes m_2) = \frac{1}{2} \left( m_1 \otimes 1 \otimes m_2, (-1)^{|m_1|.|m_2|} m_2 \otimes 1 \otimes m_1 \right) \\
&= (m_1 \otimes 1 \otimes m_2, 0)-\frac{1}{2} \left( m_1 \otimes 1 \otimes m_2, -(-1)^{|m_1|.|m_2|} m_2 \otimes 1 \otimes m_1 \right) \\
&= (m_1 \otimes 1 \otimes m_2, 0).
\end{align*}
This finishes the proof.
\end{proof}

\section{Universal enveloping algebra of free Lie algebroids}

\subsection{Structure theorem}

We consider the two pairs of adjoint functors
\[
\xymatrix@C=50pt{
\mathbf{mod}_{A / \mathbb{T}_A} \ar@<2pt>[r]^{\textit{free}} & \ar@<2pt>[l]^{\textit{forget}}
\mathbf{Lie}_{\mathbf{k}/A} \ar@<2pt>[r]^{\mathrm{U}_{\mathbf{k}/A}} & \ar@<2pt>[l]^{\mathcal{P}} \mathbf{alg}_{A/\mathbf{End}_\mathbf{k}(A)}
}
\]
Our goal is to describe explicitly the functor $\mathrm{U}_{\mathbf{k}/A}\circ \textit{free}$. For this we need to introduce some notation: let $V$ be an object of $\mathbf{mod}_{A / \mathbb{T}_A}$. For any nonnegative integer $n$, 
\begin{enumerate}
\item[--] We denote by $\Sigma_V^{[n]}$ the co-equalizer of  the $n-1$ maps from $\Sigma_V^{\otimes n-1}$ to $\Sigma_V^{\otimes n}$ induced by the map $1 \rightarrow A \rightarrow \Sigma_V$,
\item[--] We denote by $(\Sigma_V^*)^{[n]}$ the equalizer of the $n-1$ maps from $(\Sigma_V^*)^{\otimes n}$ to $(\Sigma_V^*)^{\otimes n-1}$ induced by the map $\Sigma_V^* \rightarrow A$.
\end{enumerate}

\begin{theorem} \label{prenzlauer}
Let $V$ be an object of $\mathbf{mod}_{A / \mathbb{T}_A}$. 
\begin{enumerate}
\item[(a)] There is a unique algebra structure on $\underset{n}{\varinjlim} \,\Sigma_V^{[n]} $ such that the multiplication 
is induced by the algebra structure of $\mathrm{T} \Sigma_V$. \\
\item[(b)] There is a canonical isomorphism $\underset{n}{\varinjlim} \,\Sigma_V^{[n]} \simeq \mathrm{U}_{\mathbf{k}/A}(\textit{free}(V, \rho))$ between algebra objects in $\mathbf{corr}_A^{\, nilp}$.  \\
\item[(c)] The universal enveloping algebra $\mathrm{U}_{\mathbf{k}/A}(\textit{free}(V, \rho))$ carries a natural filtration such that for any positive integer $n$:
\[
\begin{cases}
\mathrm{F}^n \mathrm{U}_{\mathbf{k}/A}(\textit{free}(V, \rho)) \simeq \Sigma_V^{[n]} \\
\mathrm{Gr}^n \mathrm{U}_{\mathbf{k}/A}(\textit{free}(V, \rho)) \simeq V^{\otimes n}.
\end{cases}
\]
Besides, $\mathrm{F}^n \mathrm{U}_{\mathbf{k}/A}(\textit{free}(V, \rho))$ lies in $\mathbf{corr}_A^{\, \leq n}$.
\end{enumerate}
\end{theorem}

\begin{remark}
All these statements admit their counterpart for the jet algebra $\mathrm{J}_{\mathbf{k}/A}(\textit{free}(V, \rho))$. It is isomorphic to $\underset{n}{\varprojlim} (\Sigma_V^*)^{[n]}$, and carries a co-filtration whose graded pieces are $\mathrm{T}^n V^*$. Here the ring structure we put on $\mathrm{T} \Sigma_V^*$ is the one given by the shuffle product (i.e. we use the cocommutative coproduct of $\mathrm{T} \Sigma_V$).
\end{remark}

\begin{proof}

(a) Let $\mathcal{I}$ be the graded sub $A$-module of the tensor algebra $\mathrm{T} \Sigma_V$ defined as follows: $\mathcal{I}^n$ is the span of all elements of the form 
\[
s_1 \otimes \ldots \otimes s_{p-1} \otimes 1 \otimes s_p \otimes \ldots \otimes s_{n-1} - s_1 \otimes \ldots \otimes s_{q-1} \otimes 1 \otimes s_q \otimes \ldots \otimes s_{n-1}
\] 
for any integers $p$ and $q$ with $1 \leq p, q \leq n$, where we see $1$ as an element of $\Sigma_V$ via the map $A \rightarrow \Sigma_V$. Then $\mathcal{I}$ is obviously a graded \footnote{Here we take the natural gradation on the tensor algebra, so elements of $\Sigma_V$ are of degree one.} two sided ideal of $\mathrm{T} \Sigma_V$, so that there is a natural multiplicative morphism
\[
\mathrm{T} \Sigma_V \rightarrow \mathrm{T} \Sigma_V / \mathcal{I}
\]
and graded pieces of $\mathrm{T} \Sigma_V / \mathcal{I}$ are exactly the $\Sigma_V^{[n]}$. Now the following diagrams
\[
\xymatrix{\Sigma_V^{[p-1]} \otimes \Sigma_V^{[q]} \ar[r] \ar[d] & \Sigma_V^{[p]} \otimes \Sigma_V^{[q]} \ar[d] \\
\Sigma_V^{[p+q-1]} \ar[r] & \Sigma_V^{[q]}
} \qquad \xymatrix{\Sigma_V^{[p]} \otimes \Sigma_V^{[q-1]} \ar[r] \ar[d] & \Sigma_V^{[p]} \otimes \Sigma_V^{[q]} \ar[d] \\
\Sigma_V^{[p+q-1]} \ar[r] & \Sigma_V^{[q]} 
}
\]
commute, which gives the first point.
\par \medskip
(b) Let $(R, p)$ be an object in $\mathbf{alg}_{A/\mathbf{End}_\mathbf{k}(A)}$ and let $\phi \colon V \rightarrow R$ be a map that commutes with the anchors. Then, for any $a$ in $A$, 
\[
\phi(v) a - (-1)^{|v| \times |a|} a \phi(v)=p(\phi(v))(a)=\rho(v)(a).
\]
We can now define a map $\Phi$ from $\Sigma_V$ to $R$ as follows: it maps $a$ to $a$ (that is to $a.1$ where $1$ is the unit of $R$), and maps $V$to $R$ via $\varphi$. We claim that this map is a morphism in $\mathbf{corr}_A$. Indeed, for the left $A$-linearity this is obvious, and for the right $A$-linearity, we have
\begin{align*}
\Phi(a', v) a &= a' a +\phi(v) a \\
&= a' a  + \rho(v)(a) + (-1)^{|v| \times |a|} a \phi(v) \\
&= \Phi((a', v) a)
\end{align*}
By the universal property of the tensor algebra, this defines a multiplicative map $\mathrm{T} \Sigma_V \rightarrow R$. Since $1$ (considered again as an element of $\Sigma_V$) is mapped to the unit of $R$, this map vanishes obviously on the graded ideal $\mathcal{I}$, and furthermore, the diagram 
\[
\xymatrix@R=10pt@C=15pt{\Sigma_V^{[n-1]} \ar[dd] \ar[rd] & \\
& R\\
\Sigma_V^{[n]} \ar[ru]}
\]
commutes. Hence we can fill in the dotted arrow
\[
\xymatrix@R=20pt@C=10pt{\underset{n}{\varinjlim} \,\Sigma_V^{[n]} \ar@{.>}[rd]& \\
V \ar[u] \ar[r] & R
}
\]
To prove the uniqueness of the lift, it suffices to prove that the algebra $\underset{n}{\varinjlim} \,\Sigma_V^{[n]}$ is generated by $V$. To do that, we take any element in this algebra, say $s_1 \otimes \ldots \otimes s_n$. We write $s_1=a+v$. Then
\begin{align*}
s_1 \otimes \ldots \otimes s_n &= a \otimes s_2 \otimes \ldots s_n + v \otimes s_2 \otimes \ldots \otimes s_n \\
&= 1 \otimes  as_2 \otimes \ldots s_n + v \otimes s_2 \otimes \ldots \otimes s_n \\
&= as_2 \otimes \ldots s_n + v . ( s_2 \otimes \ldots \otimes s_n)
\end{align*}
Applying this algorithm repeatedly, we see that $\Sigma_V^{[n]}$ is the span of elements of the form $v_1 \otimes \ldots \otimes v_n$.
\par \medskip
(c) We define the filtration on $\underset{n}{\varinjlim} \,\Sigma_V^{[n]}$ by setting the nth piece as $\Sigma_V^{[n]}$. There is a natural map $\Sigma_V^{[n]} \rightarrow \Sigma_V^{\otimes n} \rightarrow V^{\otimes n}$ given by the morphism $\Sigma_V \rightarrow V$. This map is obviously onto. Its kernel contains $\Sigma_V^{[n-1]}$, let us prove this is an equality. We take an element of the form $s_1 \otimes \ldots \otimes s_n$ in the kernel. Decomposing each $s_i$ as $a_i + v_i$, we see that 
$s_1 \otimes \ldots \otimes s_n$ can be split as the sum of $v_1  \otimes \ldots \otimes v_n$ and elements in which at least one element in every pure tensor lies in $A$. Using the same trick as in (b), these elements lie in $\Sigma_V^{[n-1]}$. This finishes the proof.
\end{proof}
%
%

\subsection{Derived construction}

In this section, we explain how the previous construction fits in the setting of derived categories. We can consider the derived category $\mathrm{D}(\mathbf{corr}_A)$ of bimodules. As a triangulated category, this is simply $\mathrm{D}(A \otimes_{\mathbf{k}} A)$ but the monoidal structure differs: for instance the unit is $A$ and not $A \otimes_{\mathbf{k}} A$.

\begin{lemma}
The functor
\begin{align*}
\mathbf{mod}_{A/\mathbb{T}_A}^{\mathrm{dual}} & \rightarrow \mathbf{corr}_A^{\mathrm{nilp}} \\
(V, \rho) & \rightarrow \mathrm{U}_{\mathbf{k}/A}(\textit{free}({V}, {\rho}))
\end{align*}
respects quasi-isomorphisms.
\end{lemma}

\begin{proof}
Let ${V}_1 \rightarrow {V}_2$ be a morphism between two perfect anchored $A$-modules. The morphism
\[
\mathrm{U}_{\mathbf{k}/A}(\textit{free}({V}_1, {\rho}_1)) \rightarrow \mathrm{U}_{\mathbf{k}/A}(\textit{free}({V}_2,  {\rho}_2))
\]
is filtered, and the corresponding graded morphism is $\mathrm{T}{V}_1 \rightarrow \mathrm{T} {V}_2$. Since the ${V}_i$ are h-flat, this is a quasi-isomorphism. 
\end{proof}

It follows from this lemma that the isomorphism class of $\mathrm{U}_{\mathbf{k}/A}(\textit{free}({V}, {\rho}))$ in $\mathrm{D}(\mathbf{corr}_A^{\mathrm{nilp}})$ only depends of the class of the anchor map $\rho$ in $\mathrm{D}(A)$.

\subsection{Formality}

We fix again $(V, \rho)$. For every $A$-module $M$, we can consider the $A$-module 
\[
M_V= \mathrm{U}_{\mathbf{k}/A}(\textit{free}\,(V, \rho)) \otimes_A M.
\]
Remark that $\mathrm{U}_{\mathbf{k}/A}(\textit{free}\,(V, \rho))$ is h-flat, so there is no need to take the derived tensor product. The $A$-module $M_V$ has a natural filtration, induced by the filtration on the universal enveloping algebra. The corresponding graded algebra is $\mathrm{T}V \otimes_A M$.
\begin{theorem}
Let $M$ be an $A$-module. Then the following properties are equivalent: 
\begin{enumerate}
\item[(a)] The filtration on $M_V$ splits in $\mathrm{D}(A)$.
\item[(b)] The filtration on $\mathrm{F}^2 M_V$ splits in $\mathrm{D}(A)$.
\item[(c)] $\Theta_M$ and $\Theta_{V \otimes M}$ vanish.
\end{enumerate}
\end{theorem}

\begin{proof} $ $ \par \medskip
(a) $\Rightarrow$ (b) obvious.
\par \medskip
(b) $\Rightarrow$ (c) the exact sequence
\[
0 \rightarrow \mathrm{F}^0  M_V \rightarrow \mathrm{F}^1 M_V \rightarrow \mathrm{Gr}^1 M_V \rightarrow 0
\]
identifies with
\[
0 \rightarrow M \rightarrow \Sigma_V \otimes_A M \rightarrow V \otimes_A M \rightarrow 0.
\]
Hence the HKR class of $M$ vanish. Let us now investigate the map
\[
\mathrm{F}^2 M_V \rightarrow \mathrm{Gr}^2 M_V.
\]
Since $\Theta_M$ vanishes, thanks to Proposition \ref{stargarder}, there exists a derived connection on $M$. If $q \colon \widetilde{M} \rightarrow M$ is a corresponding quasi-isomorphism attached to this derived connection and $\mu$ is the map \eqref{pigeonnier} introduced in the proof of \textit{loc. cit}, the map
\[
\Sigma_V \otimes_A \widetilde{M} \rightarrow M \oplus (V \otimes M) \simeq M \otimes_A \Sigma_V
\]
given by 
\[
s \otimes m \mapsto \left( \mu(s \otimes m), \pi(s) \otimes q(m) \right)
\] 
is a quasi-isomorphism. It is possible to do this construction two times, that is to find a tower $\widetilde{M}_2 \rightarrow \widetilde{M}_1 \rightarrow M$ of quasi-isomorphisms, as well as quasi-isomorphisms
\[
\Sigma_V^{\otimes 2} \otimes_A \widetilde{M}_2 \simeq \Sigma_V \otimes_A \widetilde{M}_1 \otimes_A \Sigma_V \simeq M \otimes_A \Sigma_V^{\otimes 2}.
\]
Then we can consider the diagram
\[
\xymatrix{
\Sigma_V \otimes_A \widetilde{M}_2 \ar@<2pt>[r] \ar@<-2pt>[r] \ar[d] & \Sigma_V^{\otimes 2} \otimes_A \widetilde{M}_2 \ar[d] \\
M \otimes_A \Sigma_V  \ar@<2pt>[r] \ar@<-2pt>[r]& M \otimes_A \Sigma_V^{\otimes 2}
}
\]
which implies\footnote{This argument works in fact for any $n$: we have $\Sigma_V^{[n]} \otimes_A M \simeq M \otimes_A \Sigma_V^{[n]}$ as $A$-modules.} that there is a quasi-isomorphism 
\[
\Sigma_V^{[2]} \otimes_A M \simeq M \otimes_A \Sigma_V^{[2]}.
\]
We have now a better grasp on what is going on: indeed we have a digram
\[
\xymatrix@C=15pt{
\Sigma_V^{\otimes 2} \ar[d]\ar[r]^-{\sim} & \Sigma_V \oplus (\Sigma_V \otimes_A V )\ar[r]^-{\sim}&A \oplus V \oplus  (\Sigma_V \otimes_A V)\ar[r] & A \oplus (\Sigma_V \otimes_A V) \\
\Sigma_V^{[2]} \ar@/_1pc/[urrr]_-{\sim} &&
}
\]
where the last horizontal map is $(\mathrm{id},1 \otimes (\star), -\mathrm{id})$. Putting everything together, the map
$\Sigma_V^{[2]} \otimes_A M \rightarrow V^{\otimes 2} \otimes_A M$ is isomorphic in $\mathrm{D}(A)$ to the map
\[
M \oplus (M \otimes_A \Sigma_V \otimes_A V) \rightarrow M \otimes_A V^{\otimes 2},
\] where the first component is zero. Hence this map admits a splitting if and only if $\mathrm{id}_M \otimes \Theta_V$ vanishes. Then we conclude using Proposition \ref{eberswalder}.
\par \medskip
(c) $\Rightarrow$ (a) Using Proposition \ref{eberswalder} again, we see that for any integer $k$, $\Theta_{V^{\otimes k} \otimes M}$ vanishes. Then we construct a section of the morphism 
\[
\Sigma_V^{\otimes n} \otimes _A M \rightarrow V^{\otimes n} \otimes_A M
\]
by induction as follows:
\begin{align*}
V^{\otimes n} \otimes _A M & \simeq V \otimes (V^{\otimes n-1} \otimes _A M)  \\
& \rightarrow \Sigma_V \otimes _A (V^{\otimes n-1} \otimes _A M)  &&\textrm{since}\, \Theta_{V^{\otimes n-1} \otimes M}=0 \\
& \rightarrow \Sigma_V \otimes_A ( \Sigma_V^{\otimes n-1} \otimes_A M) &&\textrm{by induction}\\
&= V^{\otimes n} \otimes _A M.
\end{align*}
This finishes the proof, since the required splitting is simply given by the composition
\[
V^{\otimes n} \otimes _A M \rightarrow \Sigma_V^{\otimes n} \otimes _A M \rightarrow \Sigma_V^{[n]} \otimes _A M.
\]
\end{proof}

\section{Square zero extensions}

\subsection{Setting}
Let $(V, \rho)$ be a perfect anchored $A$-module, which is cohomologically concentrated in positive degrees.
\begin{definition}
The $A$-augmented $\mathbf{k}$-cdga $\mathcal{A}_{V, \rho}$ is defined as follows: the underlying module is $A \oplus V^{*}[-1]$ where $V^{*}$ denotes the derived dual of $V$ over $A$, the product is the one of the usual split square zero-extension on $A$ by $V^{*}[-1]$, and the differential is
\[
d(a, \phi)= (d_A(a), \rho^*(da) - d_{V^{*}}(\phi)).
\]
\end{definition}
Then $\mathcal{A}_{V, \rho}$ is a cdga cohomologically concentrated in nonpositive degrees. The exact sequence given by the augmentation ideal is
\begin{equation} \label{titi}
0 \xrightarrow{\iota} V^*[-1] \rightarrow \mathcal{A}_{V, \rho} \xrightarrow{\pi} A \rightarrow 0.
\end{equation}

\subsection{Extension theorem}
We state and prove our main result linking non split square zero extensions and derived $(V, \rho)$-connections.
\begin{theorem}
Let $M$ be a $A$-module. Then $M$ admits a derived $(V, \rho)$ connection if and only if it is the derived pullback of an $\mathcal{A}_{V, \rho}$-module. 
\end{theorem}

\begin{proof}
Let $M$ be an $A$-module that can be written as the pullback of a flat $\mathcal{A}_{V, \rho}$-module $T$. Then there is an exact sequence of $A$-modules
\[
0 \rightarrow V^* \otimes_A M [-1] \xrightarrow{\chi} T \xrightarrow{\pi} M \rightarrow 0
\]
hence a morphism $M \rightarrow V^* \otimes_A M$ in $\mathrm{D}(\mathcal{A}_{V, \rho})$, which we want to underlie a derived $(V, \rho)$ connection on $M$. In order to achieve this, the crucial step is to endow the mapping cone of $\chi$, whose underlying $A$-module is $ T \oplus (V^* \otimes_A M)$,  with an $A$-module structure. We define it by a morphism
\begin{align*}
\vartheta \colon A \otimes_{\mathbf{k}} \mathrm{cone}(\chi) & \rightarrow \mathrm{cone}(\chi) \\
a \otimes (t, x) & \rightarrow (at, a x - \rho^*(da) \otimes \pi(t)).
\end{align*}
where $x$ is in $V^* \otimes_A M$. Let us check that $\vartheta$ is a morphism of complexes. We denote by $\delta$ the differential of $V^* \otimes_A M$ and ind the sequel, we will consider $\chi$ rather as a morphism from $V^* \otimes_A M$ to $T$ of degree $1$
that satisfies $\chi(ax)=(-1)^a \chi(x)$. Then for any $x$ of  $V^* \otimes_A M$ and any $t$ in $T$,
\begin{align*}
&\vartheta \left\{d_A(a) \otimes (t, x) + (-1)^{|a|} a \otimes(d_T(t)+\chi(x), \delta(x)) \right\} \\
&=\left\{d_A(a)t + (-1)^{|a|} (a d_T(t)+a \chi(x)), \right. \\
&\quad \left. \quad d_A(a) x - \rho^*(d [d_A(a)]) \otimes \pi(t) +(-1)^{|a|} (a \delta(x)-\rho^*(da) \otimes \pi(d_T(t)) \right\}
\end{align*}
and
\begin{align*}
d_{\mathrm{cone}(\chi)} \vartheta &\left\{a \otimes (t, x)\right\} \\
&=\left\{d_T(at)+\chi(ax - \rho^*(da)\otimes \pi(t)), \delta(ax-\rho^*(da)\otimes \pi(t)) \right\}
\end{align*}
We compute separately the two components: 
\begin{align*}
&d_T(at)+\chi(ax - \rho^*(da)\otimes \pi(t)) \\
&=d_B(a) t+(-1)^{|a|} a d_T(t) +(-1)^a \chi(x) - \chi(\rho^*(da) \otimes \pi(t)) \\
&=d_A(a) t +(-1)^{|a|} a d_T(t) +(-1)^{|a|} a \chi(x)
\end{align*}
and 
\begin{align*}
&\delta(ax+\rho^*(da)\otimes \pi(t))=d_A(a)x +(-1)^{|a|} a \delta(x) \\
& \qquad - \rho^*(d[d_A(a)]) \otimes \pi(t) - (-1)^{|a|} \rho^*(da) \otimes \pi(d_T(t)).
\end{align*}
On the other hand, $\vartheta$ is obviously associative, hence we get that it defines an $A$-action. Let us set 
$\widetilde{M}=\mathrm{cone}(\mu)$. The morphism
\[
-\pi \colon  \widetilde{M} \rightarrow M
\]
is a quasi-isomorphism. We define 
\[
\nabla \colon \widetilde{M} \rightarrow V^* \otimes_A M
\]
as the natural $\mathbf{k}$-linear morphism of complexes given by the projection on the first factor. Then we have
\begin{align*}
\nabla \{ a(t, x) \} &= \nabla (at, a x - \rho^*(da) \otimes \pi(t)) \\
&= a x - \rho^*(da) \otimes \pi(t) \\
&= a \nabla (t, x) + \rho^*(da) \otimes (-\pi)(t)
\end{align*}
so $\nabla$ defines a derived $(V, \rho)$-connection on $M$.
\par \medskip
Conversely, assumes that $M$ admits a derived $(V, \rho)$-connection $\nabla$. We can write $\nabla$ as a $\mathbf{k}$-linear morphism
\[
\nabla \colon \widetilde{M} \rightarrow V^* \otimes_A M
\]
where $q \colon \widetilde{M} \rightarrow M$ is a quasi-isomorphism. We define an $\mathcal{A}_{V, \rho}$-structure on the cone of $-\nabla[-1]$, whose underlying complex is $V^* \otimes_A M [-1] \oplus \widetilde{M}$, using the morphism
\begin{align*}
\Theta \colon A \otimes_{\mathbf{k}} \mathrm{cone}(-\nabla [-1]) & \rightarrow \mathrm{cone}(-\nabla[-1]) \\
(a, \phi) \otimes (x, \tilde{m}) & \rightarrow (ax +\phi \otimes q(\tilde{m}), a \tilde{m}).
\end{align*}
Here, contrarily to the first part of the proof, $x$ will be an element of $V^* \otimes M [-1]$ and $\delta$ will be the differential of $V^* \otimes M [-1]$.  As we did for $\chi$, it is better to see $\nabla$ as a morphism of of degree $1$ from $\tilde{M}$ to $V^* \otimes_A M[-1]$ satisfying 
\[
\nabla(a\tilde{m})=(-1)^{|a|} a \nabla(\tilde{m})+\rho^*(da) \otimes q(\tilde{m})
\]
First we check the associativity: 
\begin{align*}
(a_1&, \phi_1). \left\{ (a_2, \phi_2). (x, \tilde{m}) \right\} \\
&= (a_1, \phi_1) . (a_2 x +\phi_2 \otimes q(\tilde{m}), a_2 \tilde{m}) \\
&=(a_1 a_2 x + (a_1 \phi_2 + \phi_1 a_2) \otimes q(\tilde{m}), a_1a_2 q(\tilde{m})) \\
&= \left\{(a_1, \phi_1). (a_2, \phi_2)\right\}. (x, \tilde{m}).
\end{align*}
Now we must check that $\Theta$ respects the differentials, which follows from a direct calculation: 
\begin{align*}
&\Theta  \left\{ d_B (a, \phi)  \otimes  (x, \tilde{m})+((-1)^{|a|} a, (-1)^{|\phi|} \phi) \otimes d(x, \tilde{m}) \right\} \\
&=\Theta \left\{(d_A(a), \rho^*(da)-d_{V^*}(\phi)) \otimes  (x, \tilde{m}) \right\}\\
&\quad +\Theta \left\{((-1)^{|a|} a, (-1)^{|\phi|} \phi) \otimes (\delta(x)+ \nabla \tilde{m}, d_{\widetilde{M}} (\tilde{m}))\right\} \\
&=\left(d_A(a) x + (\rho^*(da)-d_{V^*}(\phi)) \otimes q(\tilde{m})+ (-1)^{|a|}a (\delta(x)+ \nabla \tilde{m}) \right. \\
& \quad + \left.(-1)^{|\phi|} \phi \otimes q(d_{\widetilde{M}} (\tilde{m}))), d_A(a) \tilde{m}+(-1)^{|a|} d_{\widetilde{M}} (\tilde{m})) \right) \\
&=(\delta(ax)+\delta(\phi \otimes q(\tilde{m})) + \rho^*(da) \otimes q(\tilde{m})+(-1)^{|a|} a\nabla{\tilde{m}}, d_{\widetilde{M}}(a \tilde{m}))\\
&= (\delta(ax)+\delta(\phi \otimes q(\tilde{m})) + \nabla (a\tilde{m}), d_{\widetilde{M}}(a \tilde{m})) \\
&= d_{\mathrm{cone}(-\nabla [-1])} (ax + \phi \otimes q(\tilde{m}), a \tilde{m}) \\
&= d_{\mathrm{cone}(-\nabla [-1])}  \Bigl( \Theta \{(a, \phi) \otimes (x, \tilde{m})\} \Bigr).
\end{align*}
Let $T$ be the cone of $-\nabla [-1]$ endowed with the $\mathcal{A}_{V, \rho}$-structure given by $\Theta$. There is an exact sequence of $\mathcal{A}_{V, \rho}$-modules
\[
0 \rightarrow V^*[-1] \otimes M \rightarrow T \rightarrow \widetilde{M} \rightarrow 0.
\]
\end{proof}

\subsection{Koszul duality}
In all this section, we will make the following hypotheses: 
\begin{enumerate}
\item[--] $A$ is a cofibrant cdga concentrated in nonpositive degrees.
\item[--] $V$ is a perfect anchored $A$-module concentrated in positive degrees (in particular, the cdga $\mathcal{A}_{V, \rho}$ is also concentrated in nonpositive degrees).
\end{enumerate}

\par \medskip

The derived tangent complex $\mathbb{T}_{A / \mathcal{A}_{V, \rho}}$ classifies derived deformations of $A$ as an $\mathcal{A}_{V, \rho}$-module. Explicitly, we have
\[
\mathbb{T}_{A / \mathcal{A}_{V, \rho}}=\mathbf{Der}_{\mathcal{A}_{V, \rho}}(\widetilde{A}, \widetilde{A})
\] 
where $\widetilde{A}$ is a quasi-isomorphic replacement of the $\mathcal{A}_{V, \rho}$-algebra $A$, such that the map $\mathcal{A}_{V, \rho} \rightarrow \widetilde{A}$ is a cofibration. A crucial observation made in \cite{calaque_lie_2014} is that $\mathbb{T}_{A / A_{V, \rho}}$ is naturally a derived\footnote{We assumed at the beginning that $A$ was a cofibrant cdga, so our dg-Lie algebroids are already derived.} $(\mathbf{k}, A)$ Lie algebroid, with the bracket given by the usual bracket of derivations, and the anchor map being 
\[
\mathbf{Der}_{\mathcal{A}_{V, \rho}}(\widetilde{A}, \widetilde{A}) \rightarrow \mathbf{Der}_{\mathbf{k}}(\widetilde{A}, \widetilde{A}) =\mathbb{T}_{\widetilde{A}/{\mathbf{k}}}
\]
The main result linking square zero extensions and anchored modules is the following: 
\begin{proposition}
Let $(V, \rho)$ be a perfect anchored $A$-module, and assume that $A$ is a cofibrant cdga concentrated in nonpositive degrees. Then there exists a natural quasi-isomorphism
\[
\textit{free}(V, \rho) \rightarrow \mathbb{T}_{A/\mathcal{A}_{V, \rho}}
\]
of $(\mathbf{k}, A)$ dg-Lie algebroids.
\end{proposition}

\begin{proof}
This result follows from \cite[Corollary 5.2.2]{MR3701353}, and is well known to experts in derived deformation theory. For the sake of completeness, we explain concretely where the morphism comes from. To lighten notation, we put $B=\mathcal{A}_{V, \rho}$. 
\par \medskip
We introduce the $B$-algebra $A^{\dagger}$ which is the cone of the morphism 
\[
\iota \colon V^*[-1] \rightarrow B
\] 
endowed with the product
\[
\tau[(b, \phi) \otimes_{\mathbf{k}} (b', \phi')]=(bb', \pi(b) \phi' + \phi \pi(b)).
\]
We will consider $\iota$ as a morphism from $V^*$ to $B$ of degree $-1$ satisfying $\iota(b\phi)=(-1)^b b \iota(\phi)$.
Let us check that $\tau$ is a morphism of complexes. We compute separately the two components of
\begin{align*}
\tau[(d_B(b) +& \iota(\phi) , d_{V^*}(\phi)) \otimes (b', \iota(\phi'))] \\
&+\tau[((-1)^{|b|} b, (-1)^{|\phi|} \phi) \otimes (d_B(b') + \iota(\phi) , d_{V^*}(\phi'))]
\end{align*}
\begin{align*}
\textrm{\ding{62}} \, &\textrm{First component} \\
&(d_B(b) + \iota(\phi)) b' + (-1)^b  b (d_B(b') + \iota(\phi)) \\
&= d_B(bb') + \iota(\phi) b' + (-1)^b b \iota(\phi) \\
& d_B(bb')+\iota(\phi b'+ b \phi)\\
& \\
\textrm{\ding{62}} \, &\textrm{Second component} \\
&\pi(d_B (b)) \phi' + d_{V^*}\phi \pi(b')+(-1)^{|b|} \pi(b) d_{V^*}(\phi')  + (-1)^{|\phi|} \phi \pi(d_B(b')) \\
&= d_{V^*}(\pi(b) \phi' + \phi \pi(b')).
\end{align*}
Remark that the morphism $A^{\dagger} \rightarrow A$ is a $B$-linear quasi-isomorphism.
Let us now consider the morphism
\[
\theta \colon V \rightarrow \mathrm{Der}_B(A^{\dagger}, A)
\]
given by 
$
\theta_v (b, \phi)=-\phi(v) 
$.
We have 
\[
\begin{cases}
\theta_v(b'(b, \phi))= \theta_v(b'b, \pi(b) \phi)=-\pi(b) \phi(v) \\
\theta_v [(b, \phi)(b', \phi')]=-(\pi(b) \phi'+ \phi \pi(b'))(v)=-\pi(b) \phi'(v)-\phi(v) \pi(b')
\end{cases}
\]
so $\theta_v$ is a $B$-linear derivation of $A^{\dagger}$. The algebra $A^{\dagger}$ carries another important feature: the quasi-isomorphism $A^{\dagger} \rightarrow A$ admits a $\mathbf{k}$-linear splitting $\sigma$ given by 
\[
\sigma(a) = (a, -\rho^*(da))
\]
We must check that $\sigma$ commutes with the differentials and is multiplicative.
\begin{align*}
\sigma(a) \sigma(a') &= (a, -\rho^*(da)). (a', -\rho^*(da')) = (aa', -a \rho^*(da')-\rho^*(da) a')\\
&=(aa', -\rho^*(d(aa')))= \sigma(aa')
\end{align*}
and
\begin{align*}
\sigma(d_A(a)) &=(d_A(a), -\rho^*(d(d_A(a)))=(d_A(a), -d_{V^*}(\rho^*(da))) \\
&=(d_B(a)-\iota(\rho^*(da)), -d_{V^*}(\rho^*(da))) \\
&= d_{A^{\dagger}}(\sigma(a)).
\end{align*}
We can assume that the cofibrant resolution $\widetilde{A}$ of $A$ as a $B$-algebra dominates $A^{\dagger}$. Let us consider the diagram
\[
\xymatrix{
 & \mathbf{Der}_{B}(A^{\dagger}, A)  \ar[d] \ar[r] &\mathbf{Der}_B(\widetilde{A}, A) \\
V\ar[rd]_-{\rho} \ar[ru]^-{\theta} & \mathbf{Der}_{\mathbf{k}}(A^{\dagger}, A) \ar[d]^-{\circ \sigma} & \mathbf{Der}_{B}(\widetilde{A}, \widetilde{A}) \ar[u] \ar[d]\\
& \mathbf{Der}_{\mathbf{k}}(A, A) & \mathbf{Der}_{\mathbf{k}}(\widetilde{A}, \widetilde{A}) \ar@{-}_-{\sim}[l]
} 
\]
where for any morphism of cdgas $D \rightarrow C$ and any $C$-module $M$ we put $\mathbf{Der}_C(D, M)=\mathrm{RHom}_C(\mathbb{L}_{C/D}, M)$.
The commutativity follows from the fact that
\[
\theta_v(\sigma(a))=\theta_v(a, -\rho^*(da))=\rho^*(da) (v)=\rho(v)(a).
\]
Hence, after replacing $V$ by a quasi-isomorphic $\widetilde{A}$-module, we get a commutative diagram
\[
\xymatrix{ V \ar[rr]^{\tilde{\theta}} \ar[rd] &&   \mathbf{Der}_{B}(\widetilde{A}, \widetilde{A}) \ar[ld] \\
& \mathbf{Der}_{\mathbf{k}}(\widetilde{A}, \widetilde{A}) & 
}
\]
The right arrow is the anchor of the $(\mathbf{k}, \tilde{A})$ Lie algebroid $\mathbf{Der}_{B}(\widetilde{A}, \widetilde{A})$, so we get a morphism
\[
\textit{free}(V, \rho) \rightarrow \mathbf{Der}_{B}(\widetilde{A}, \widetilde{A}).
\] 
\end{proof}

\begin{theorem}
There are natural multiplicative quasi-isomorphisms
\[
\begin{cases}
\mathrm{U}_{\mathbf{k}/A}(\textit{free}(V, \rho)) \rightarrow \mathrm{RHom}_{\mathcal{A}_{V, \rho}}(A, A) \\
A \overset{\mathbb{L}}{\otimes}_{\mathcal{A}_{V, \rho}} A \rightarrow \mathrm{J}_{\mathbf{k}/A}(\textit{free}(V, \rho))
\end{cases}
\]
of algebra objects in $\mathrm{D}(\mathbf{corr}_A^{\mathrm{nilp}})$.
\end{theorem}

\begin{proof}
Let us construct the first morphism, the other one being dual. We keep the same notation as in the previous proof, and we consider the chain of morphisms
\[
\textit{free}(V, \rho) \rightarrow \mathbf{Der}_{B}(\widetilde{A}, \widetilde{A}) \rightarrow \mathrm{Hom}_B(\widetilde{A}, \widetilde{A}).
\]
The algebra $\mathrm{Hom}_B(\widetilde{A}, \widetilde{A})$ is anchored over $\mathbf{End}_{\mathbf{k}}(\widetilde{A})$ and we see that the subspace $\mathbf{Der}_{B}(\widetilde{A}, \widetilde{A})$ is obviously inside the primitive part, so by the universal property of the enveloping algebra of a Lie algebroid, we get a morphism
\begin{equation} \label{sabre}
\Psi \colon \mathrm{U}_{\mathbf{k}/\widetilde{A}} (\textit{free}(V, \rho)) \rightarrow \mathrm{U}_{\mathbf{k}/\widetilde{A}}(\mathbf{Der}_{B}(\widetilde{A}, \widetilde{A})) \hookrightarrow \mathrm{Hom}_{B}(\widetilde{A}, \widetilde{A}).
\end{equation}
The last map is injective, its image is the $B$-module $\mathbf{Diff}_B(\widetilde{A})$ of relative differential operators over $B$ with coefficients in $\widetilde{A}$, which is an explicit model for the universal enveloping algebra of $\mathbf{Der}_B(\widetilde{A}, \widetilde{A})$. 
\par \medskip
It is fairly easy to understand the left action of $V$ on $\mathrm{Hom}_{B}(\widetilde{A}, \widetilde{A})$ at the level of derived categories: it is simply given by the map
\[
V \otimes_A \mathrm{RHom}^{\ell}_{B}(A, A) \simeq \mathrm{RHom}^{\ell}_{B}(V^*, A) \rightarrow \mathrm{RHom}^{r}_{B}(A, A)
\]
where the last arrow is obtained by precomposing by the connection morphism $A \rightarrow V^*$ in $\mathrm{D}(B)$ attached to \eqref{titi}, and $\mathrm{RHom}^{\ell}_B(\star, A)$ is the derived functor of  $\mathrm{Hom}_B(\star, A)$ which goes from $B$-modules to $A$-modules. Using again \eqref{titi}, there is an isomorphism
\[
\left(V \otimes_A \mathrm{RHom}^{\ell}_{B}(A, A)\right) \oplus A \xrightarrow{\sim} \mathrm{RHom}^{\ell}_{B}(A, A)
\]
in $\mathrm{D}(A)$ which means that the map
\begin{align*}
\left( V \otimes_{\widetilde{A}} \mathrm{Hom}_B(\widetilde{A}, \widetilde{A}) \right) \oplus \widetilde{A} &\rightarrow \mathrm{Hom}_B(\widetilde{A}, \widetilde{A}) \\
(v \otimes f, \tilde{a}) & \rightarrow f \circ \tilde{\theta}_v + r_{\tilde{a}}
\end{align*}
is a quasi-isomorphism of right $\widetilde{A}$-modules. Since $V$ is concentrated in positive degrees, this implies that for any integer $p$, there are quasi-isomorphisms
\[
\xymatrix{
\tau^{\leq p+1} \left( V \otimes_{\widetilde{A}} \mathrm{Hom}_B(\widetilde{A}, \widetilde{A}) \right) \oplus \widetilde{A} \ar[r] \ar[d] & \tau^{\leq p+1} \mathrm{Hom}_B(\widetilde{A}, \widetilde{A}) \\
\tau^{\leq p+1} \left( V \otimes_{\widetilde{A}} \tau^{\leq p }\mathrm{Hom}_B(\widetilde{A}, \widetilde{A}) \right) \oplus \widetilde{A} &
}
\]
The very same property folds for $\textit{free}(V, \rho)$. This implies by induction on $p$ that for any $p$ the truncated map $\tau^{\leq p} \Psi$ is a quasi-isomorphism: indeed, for negative $p$, the truncated map is zero since both members are concentrated in nonnegative cohomological degree.
\end{proof}
\section{Geometric applications}

\subsection{Global results}
The constructions of the previous sections can be globalized over an arbitrary derived scheme, thanks to descent theory for quasicoherent complexes on derived schemes (\textit{see} \cite{toen_derived_2012} and \cite[\S 2.2]{ toen_homotopical_2008}). Let $X$ be a derived scheme over $\mathbf{k}$ and let $\mathcal{V}$ be a perfect object of $\mathrm{L}_{qcoh}(X)$, that is an object of $\mathrm{L}_{parf}(X)$, cohomologically concentrated in positive degrees, and endowed with an anchor map $\rho \colon \mathcal{V} \rightarrow \mathbb{T}_X$, where $\mathbb{T}_X$ is the tangent complex of $X$. Then all the results we prove admit a global version:
\par \medskip
\fbox{Structure theorem}
\begin{enumerate}
\item[--] The algebra $\mathrm{U}(\textit{free}(V, \rho))$ admits a filtration whose pieces of level $n$ is the co-equalizers $\Sigma_{\mathcal{V}}^{[n]}$, and the corresponding graded algebra is the tensor algebra $\mathrm{T}\mathcal{V}$.
\item[--] The Jet algebra $\mathrm{J}(\textit{free}(V, \rho))$ admits a cofiltration whose pieces of degree $n$ are the equalizers $(\Sigma_{\mathcal{V}}^*)^{[n]}$.
\end{enumerate}
\par \medskip
\fbox{HKR class and derived connections}
\begin{enumerate}
\item[--] Any object $\mathcal{F}$ of $\mathrm{L}_{qcoh}(X)$ admits an HKR class $\Theta_{\mathcal{F}}$, which is a morphism from $\mathcal{F}$ to $\mathcal{V}^* [1] \otimes \mathcal{F}$ in $\mathrm{D}(X)$.
\item[--] The class $\Theta_{\mathcal{F}}$ vanishes if and only if $\mathcal{F}$ admits a derived $(\mathcal{V}, \rho)$ connection.
\end{enumerate}
\par \medskip
\fbox{Formality theorem}$ $ \par \smallskip
Let $\mathcal{F}$ be an object of $\mathrm{L}_{qcoh}(X)$.
\begin{enumerate}
\item[--] The filtration on $\mathrm{U}(\textit{free} (V, \rho)) \otimes_{\mathcal{O}_X} \mathcal{F}$ splits in $\mathrm{D}(X)$ if and only if $\Theta_{\mathcal{F}}$ and $\Theta_{\mathcal{F} \otimes \mathcal{V}}$ vanish.
\item[--] The cofiltration on $\mathrm{J}(\textit{free} (V, \rho)) \otimes_{\mathcal{O}_X} \mathcal{F}$ splits in $\mathrm{D}(X)$ if and only if $\Theta_{\mathcal{F}}$ and $\Theta_{\mathcal{F} \otimes \mathcal{V}}$ vanish.
\end{enumerate}
\par \medskip
\fbox{Square zero extensions}
\begin{enumerate}
\item[--] There is a natural derived scheme $X_{\mathcal{V}, \rho}$, which is a square zero extension of $X$.
\item[--] We have multiplicative isomorphisms
\[
\begin{cases}
\mathrm{U}_{\mathbf{k}/X}(\textit{free}(\mathcal{V}, \rho)) \simeq \mathrm{RHom}_{\mathcal{O}_{X_{\mathcal{V}, \rho}}}(\mathcal{O}_X, \mathcal{O}_X) \\
\mathrm{J}_{\mathbf{k}/X}(\textit{free}(\mathcal{V}, \rho)) \simeq \mathcal{O}_X \overset{\mathbb{L}}{\otimes}_{{\mathcal{O}_{X_{\mathcal{V}, \rho}}}} \mathcal{O}_X
\end{cases}
\]
 in $\mathrm{D}^+_{\Delta_X}(X \times X)$ and $\mathrm{D}^-_{\Delta_X}(X \times X)$ respectively.
\end{enumerate}
\par \medskip
\fbox{Extension theorem}
\begin{enumerate}
\item[\vphantom{--}] For any object $\mathcal{F}$ of $\mathrm{L}_{qcoh}(X)$, the class $\Theta_{\mathcal{F}}$ vanishes if and only if $\mathcal{F}$ is the pullback of an object in $\mathrm{L}_{qcoh}(X_{\mathcal{V}, \rho})$.
\end{enumerate}
\par \medskip

\subsection{Geometric setting and quantized cycles}
The first setting is the one of \cite{bigpaper}: $S$ is a locally trivial first order thickening of a smooth $\mathbf{k}$-scheme $X$ by a locally free sheaf $\mathcal{I}$. There is an associated exact sequence
\[
0 \rightarrow \mathcal{I} \rightarrow \mathcal{E} \rightarrow \Omega^1_X \rightarrow 0
\] 
where $\mathcal{E}=(\Omega^1_S)_{|X}$, whose extension class $\eta$ in $\mathrm{Ext}^1_{\mathcal{O}_X}(\Omega^1_X, \mathcal{I})$ determines entirely $S$. Hence the sheaf
\[
\mathcal{V}=\mathrm{cone}\,(\mathrm{T}_X \rightarrow \mathcal{E}^*) [-1]
\]
is naturally endowed with an anchor given by the diagram
\[
\xymatrix{\mathrm{T}_X \ar[d] \ar[r] & \mathcal{E}^* \\
\mathrm{T}_X&
}
\]
Note that $\mathcal{V}$ is quasi-isomorphic to $\mathcal{I}^*[-1]$. This setting appears often in the geometric situation, which is slightly more restrictive, where $X$ is a closed subscheme of an ambient scheme $Y$, and $S$ is the first formal neighborhood of $X$ in $Y$. In that case $\mathcal{E}=\Omega^1_{Y | X}$ so
\[
\mathcal{V}= \mathrm{cone}\,(\mathrm{T}_X \rightarrow \mathrm{T}_{Y|X}) [-1].
\]
In this setting, the notion of \textit{quantized analytic cycle} introduced in \cite{grivaux_hochschild-kostant-rosenberg_2014} becomes transparent: by definition, $(X, \sigma)$ is quantized in $Y$ if the conormal exact sequence of the pair $(X, Y)$ splits globally, and $\sigma$ is such a splitting. It follows that such a $\sigma$ determines a morphism from $(\mathcal{V}, \mathrm{pr_1})$ to the module $\mathrm{N}_{X/Y}[-1]$ endowed with the null anchor, and this morphism is a quasi-isomorphism. Therefore, the quantized HKR isomorphism of \cite{grivaux_hochschild-kostant-rosenberg_2014} is induced by the canonical isomorphism
\[
\mathcal{O}_X \overset{\mathbb{L}}{\otimes}_{\mathcal{O}_S} \mathcal{O}_X \simeq \mathrm{J}(\textit{free}(\mathrm{N}_{X/Y}[-1])) \simeq \mathrm{T} (\mathrm{N}^*_{X/Y}[1]).
\]
Outside the quantized world, one of the only known result is the calculation of the derived tensor product $\mathcal{O}_X \overset{\mathbb{L}}{\otimes}_{\mathcal{O}_S} \mathcal{O}_X$ done in \cite{bigpaper}. Unwrapping what it means, we see that the description of the truncations obtained as equalizers can be interpreted as the isomorphisms
\[
\mathcal{O}_X \overset{\mathbb{L}}{\otimes}_{\mathcal{O}_S} \mathcal{O}_X \simeq \mathrm{J}_{\mathbf{k}/X}(\textit{free}(\mathcal{V}, \rho)) \simeq \varprojlim_{n \in \mathbb{N}} (\Sigma_{\mathcal{V}}^*)^{[n]}.
\]
\subsection{Deformation theory}

In \cite{arinkin_when_2012}, the authors define the HKR class of a vector bundle $\mathcal{H}$ on $X$ as the morphism
\[
\alpha_{\mathcal{H}}\colon \mathcal{H} \rightarrow \Omega^1_X [1] \otimes \mathcal{H} \rightarrow \mathrm{N}^*_{X/Y}[2] \otimes \mathcal{H}
\]
in $\mathrm{D}(X)$, where the first map is given by the Atiyah class of $\mathcal{H}$ and the second one is the extension class attached to the conormal exact sequence. Then they prove: 
\begin{enumerate}
\item[--] $\alpha_{\mathcal{H}}$ vanishes exactly when $\mathcal{H}$ extends to a locally free sheaf on the first formal neighborhood $S$ of $X$ in $Y$. 
\item[--] $\mathcal{O}_X \overset{\mathbb{L}, r}{\otimes}_{\mathcal{O}_S} \mathcal{H} \simeq \mathrm{T}(\mathrm{N}^*_{X/Y}[-1]) \otimes \mathcal{H}$ if and only if the HKR classes of $\mathcal{H}$ and $\mathrm{N}_{X/Y} \otimes \mathcal{H}$ vanish\footnote{The superscript ``r" means that the tensor product is derived with respect to the right variable.}.
\end{enumerate}
This theorem can be recovered from our construction by considering the anchored module $(\mathcal{V}, \rho)=(\mathrm{cone}\,(\mathrm{T}_X \rightarrow \mathrm{T}_{Y|X}) [-1], \mathrm{pr}_1)$. To see this we argue as follows: 
\par \medskip
First we check that the dual of $\alpha_{\mathcal{H}}$ is $\Theta_{\mathcal{H}^*}$. Indeed, $\alpha_{\mathcal{H}}$ can be described as the extension class associated with the exact sequence
\[
0 \rightarrow \mathrm{N}^*_{X/Y} \otimes \mathcal{H} \rightarrow
\mathrm{cone} \left(\Omega^1_{Y|X} \otimes \mathcal{H} \rightarrow \mathrm{P}^1_X(\mathcal{H}) \right) \rightarrow \mathcal{H} \rightarrow 0
\]
This sequence is quasi-isomorphic to 
\[
\xymatrix{
0 \ar[r] &  \mathrm{cone} \left(\Omega^1_{Y|X} \otimes \mathcal{H} \rightarrow \Omega^1_X \otimes \mathcal{H} \right) \ar[d] && \\
& \mathrm{cone} \left(\Omega^1_{Y|X} \otimes \mathcal{H} \rightarrow \mathrm{P}^1_X(\mathcal{H}) \right) \ar[r] & \mathcal{H} \ar[r] & 0
}
\]
which is
\[
0 \rightarrow \mathcal{V} \otimes \mathcal{H} \rightarrow \Sigma_{\mathcal{V}}^* \otimes \mathcal{H} \rightarrow \mathcal{H} \rightarrow 0.
\]
Using Proposition \ref{salaud} as well as the fact that $\Sigma_V^*$ is symmetric, the dual of this sequence is 
\[
0 \rightarrow \mathcal{H}^* \rightarrow \Sigma_{\mathcal{V}} \otimes \mathcal{H}^* \rightarrow \mathcal{H}^* \otimes V \rightarrow 0
\]
so the corresponding extension morphism is $\Theta_{\mathcal{H}^*}$. Hence 
\[
\alpha_{\mathcal{H}}=0 \Leftrightarrow \Theta_{\mathcal{H}^*}=0 \Leftrightarrow \Theta_{\mathcal{H}}=0
\] 
because derived connections on $\mathcal{H}$ and $\mathcal{H}^*$ are in bijection.
\par \medskip
The next step is the explicit description of $X_{\mathcal{V, \rho}}$. The underlying set is $X$, and the sheaf of rings is the total complex of
\[ \xymatrix{ & \mathcal{O}_X \ar[d]^-{d} \\
\Omega^1_{Y|X} \ar[r] & \Omega^1_X
}
\]
and the ring strucure is the same as the one on the square zero extension of $\mathcal{O}_X$ by complex given by the bottom line. There is a natural ring quasi-isomorphism from $S$ to $X_{\mathcal{V}, \rho}$ given by
\[
f \rightarrow (df_{|X}, f_{|X}).
\]
Hence $\mathrm{L}_{qcoh}(S) \simeq \mathrm{L}_{qcoh}(X_{V, \rho})$, and the result follows from the extension theorem. For the second point, it suffices to use that $\mathcal{O}_X \overset{\mathbb{L}}{\otimes}_{\mathcal{O}_S} \mathcal{O}_X$ is isomorphic to the jet algebra of $\textit{free}(\mathcal{V}, \rho)$. Then we apply the formality theorem.

\bibliographystyle{plain}
\bibliography{biblio}

\end{document}